\documentclass[12pt,a4paper]{article}

\usepackage{graphicx}
\usepackage{subfig}
\usepackage{amsmath,amsthm,amssymb}
\usepackage{esint}
\usepackage{mathrsfs}
\usepackage{appendix}
\usepackage{enumitem}

\usepackage{hyperref}

\usepackage[left=2.2cm,right=2.2cm,top=3cm,bottom=3.5cm]{geometry}

\usepackage{color}
\usepackage{array}
\usepackage{url}

\usepackage{float}

\newtheorem{theorem}{Theorem}[section]
\newtheorem*{theorem*}{Theorem}
\newtheorem{lemma}[theorem]{Lemma}
\newtheorem{corollary}[theorem]{Corollary}

\theoremstyle{definition}

\theoremstyle{remark}
\newtheorem{remark}[theorem]{Remark}
\usepackage{authblk}

\usepackage{mathtools}
\usepackage{amssymb}
\usepackage{appendix}
\usepackage{multirow}

\numberwithin{equation}{section}

\usepackage[numbers,sort&compress]{natbib}

\newcommand{\dx}{\ {\rm d} {x}}
\newcommand{\dt}{\ {\rm d} t }

\DeclareMathOperator{\Tr}{Tr}

\allowdisplaybreaks

\begin{document}

\title{Global well-posedness of the elastic-viscous-plastic sea-ice model with the inviscid Voigt-regularisation}

\author{Daniel W. Boutros\footnote{Department of Applied Mathematics and Theoretical Physics, University of Cambridge, Cambridge CB3 0WA UK. Email: \textsf{dwb42@cam.ac.uk}}\space,\space Xin Liu\footnote{Department of Mathematics, Texas A\&M University, College Station, TX 77843-3368, USA. Email: \textsf{xliu23@tamu.edu}}\space,\space Marita Thomas\footnote{Department of Mathematics and Computer Science, Freie Universität Berlin, Arnimallee 9, 14195 Berlin, Germany; Weierstrass Institute for Applied Analysis and Stochastics, Mohrenstr. 39, 10117 Berlin, Germany. Email: \textsf{marita.thomas@fu-berlin.de}}\space, and Edriss S. Titi\footnote{Department of Applied Mathematics and Theoretical Physics, University of Cambridge, Cambridge CB3 0WA UK; Department of Mathematics, Texas A\&M University, College Station, TX 77843-3368, USA; also Department of Computer Science and Applied Mathematics, Weizmann Institute of Science, Rehovot 76100, Israel. Emails: \textsf{Edriss.Titi@maths.cam.ac.uk} \; \textsf{titi@math.tamu.edu} \; \textsf{edriss.titi@weizmann.ac.il}}}

\date{January 16, 2026}

\maketitle

\begin{abstract}
In this paper, we initiate the rigorous mathematical analysis of the elastic-viscous-plastic (EVP) sea-ice model, which was introduced in \textit{[E. C. Hunke and J. K. Dukowicz, J. Phys. Oceanogr., 27, 9 (1997), 1849--1867]}. The EVP model is one of the standard and most commonly used dynamical sea-ice models. We study a regularised version of this model. In particular, we prove the global well-posedness of the EVP model with the inviscid Voigt-regularisation of the evolution equation for the stress tensor. Due to the elastic relaxation and the Voigt regularisation, we are able to handle the case of viscosity coefficients without cutoff, which has been a major issue and a setback in the computational study and analysis of the related Hibler sea-ice model, which was originally introduced in \textit{[W. D. Hibler, J. Phys. Oceanogr., 9, 4 (1979), 815–-846]}. It is worth noting that the EVP model shares some structural characteristics with the Oldroyd-B model and related models for viscoelastic non-Newtonian complex fluids. 
\end{abstract}


\vspace{0.1cm} \noindent \textbf{Mathematics Subject Classification:} 86A40 (primary), 35A01, 35A02, 35B65, 35Q86, 74D10, 74H20, 74H25, 74H30, 86A08 (secondary) 

\section{Introduction}
\subsection{Formulation of the EVP model}

Sea-ice dynamics has a strong influence on the global climate, for instance, through heat exchange with the atmosphere and the ocean around the Arctic area as well as the reflection of sunlight. Despite the development of sea-ice modelling over many decades, its complexity continues to pose significant challenges for numerical simulations of the sea-ice cover (cf. \cite{golden}). For these reasons, an enhanced theoretical understanding of the fundamental dynamical sea-ice models and their rigorous mathematical analysis is of great importance and substantial interest.

Over several decades there have been significant advances in the rigorous mathematical analysis of oceanic and atmospheric dynamics. In contrast, there have been relatively few works which obtain rigorous results in the mathematical study of sea-ice dynamics. It is the purpose of this paper to commence the rigorous mathematical analysis of the elastic-viscous-plastic (EVP) dynamical sea-ice model. 

The EVP model was introduced in \cite{hunke} as a numerical regularisation of the Hibler model from \cite{hibler}. The latter model describes sea-ice dynamics by using a viscous-plastic rheology, and the viscosity becomes singular for small strain rates and degenerate for large strain rates. 

We will denote the symmetric part of the gradient $ D(u) $ by 
\begin{equation}\label{def:symty-gradient}
    D(u) \coloneqq \frac{1}{2} \lbrack \nabla u + (\nabla u)^\top \rbrack.
 \end{equation}
The strain rate is given by 
\begin{equation}
\label{originalstrainrate}
\begin{aligned}
\overline{\mathcal D} \coloneqq \biggl( \frac{2}{\overline e^2} \bigg\vert D(u) - \frac{1}{2} [\Tr D(u)] \mathbb I_2 \bigg\vert^2 + \big\vert \Tr D(u) \big\vert^2 \biggr)^{1/2}. 
\end{aligned}
\end{equation}
Here $ u : \mathbb{T}^2 \times [0,T] \rightarrow \mathbb{R}^2$ is the unknown velocity and $ \overline e > 1 $ is a constant, representing the ratio of the major axis to minor axis lengths of the elliptic yield curve. Note that throughout this paper we will use periodic boundary conditions and take our domain to be the two-dimensional flat unit torus $\mathbb{T}^2$ (i.e. $[0,1]^2$). 

Then, the momentum equations of the Hibler sea-ice model \cite{hibler} can be written as
\begin{subequations}\label{sys:hibler}
\begin{gather}
\label{eq:hibler-momentum}
m \partial_t u + m u \cdot \nabla u = \nabla \cdot \sigma + \mathcal{T}_a + \mathcal{T}_w + m \Omega u^\perp - m g \nabla H_0, \\
\intertext{with the constitutive law for the stress tensor given by}
\dfrac{\overline e^2 \overline{\mathcal D}}{P} (\sigma - \frac{1}{2} \Tr \sigma \mathbb I_2) + \dfrac{\overline{\mathcal D}}{2P} \Tr \sigma \mathbb I_2 + \dfrac{\overline{\mathcal D}}{2} \mathbb I_2  = D(u).\label{eq:hibler-strain-rate}
\end{gather}
Here $\sigma : \mathbb{T}^2 \times [0,T] \rightarrow \mathbb{R}^{2 \times 2}$ is the stress tensor, and $ P > 0, m > 0 $, and $ \Omega $ are given constants which denote the pressure, the mass per unit area and the coefficient of the Coriolis force, respectively. The functions $\mathcal T_a$ and $\mathcal T_w$ denote the wind and ocean stresses, and are defined in equations \eqref{def:windstress} and \eqref{def:oceanstress} below, respectively. The given function $ H_0$ denotes the ocean surface topography. We also recall the notation $u^\perp = (-u_2, u_1)^\top$. In order to fix ideas, we have omitted the thermodynamical equations (i.e. we assume the mean ice thickness $h$ and the ice compactness $A$ to be constant) and purely focus on the dynamical part of the EVP model (as is also done for the EVP model as originally introduced in the paper \cite{hunke}). Furthermore, we assume that $ P $ and $m$ are given constants (as we have taken $h$ and $A$ to be constant). The analysis of the EVP model coupled with the thermodynamical equations is a subject of future work. 

One can easily check from \eqref{eq:hibler-strain-rate} that $ \sigma $ becomes singular as the strain rate $ \overline{\mathcal D} $ (see \eqref{originalstrainrate}) approaches zero, and becomes degenerate as $ \overline{\mathcal D} $ approaches infinity. This special structure, in particular, causes significant difficulties for both the analysis as well as the numerical simulations. In fact, in Hibler's original paper \cite{hibler}, a cutoff of $\overline{\mathcal D}$ is introduced to avoid this very problem. 

\smallskip 

In order to facilitate the usage of explicit numerical schemes and therefore also of parallel computing, it was proposed in \cite{hunke} to modify the viscous-plastic sea-ice rheology \eqref{eq:hibler-strain-rate} by adding an artificial elastic component. That is, one replaces the diagnostic, constitutive equation \eqref{eq:hibler-strain-rate} by the prognostic, dynamically relaxed equation for the stress tensor
\begin{equation} \label{eq:evp-strain-rate}
    \dfrac{1}{\mathcal E} \partial_t \sigma + \dfrac{\overline e^2 \overline{\mathcal D}}{P} (\sigma - \frac{1}{2} \Tr \sigma \mathbb I_2) + \dfrac{\overline{\mathcal D}}{2P} \Tr \sigma \mathbb I_2 + \dfrac{\overline{\mathcal D}}{2} \mathbb I_2  = D(u),
\end{equation}
where $ \mathcal E > 0 $ is the (artificial) elastic modulus, which is constant. Formally, by sending $ \mathcal E \rightarrow \infty $ in \eqref{eq:evp-strain-rate}, one can recover \eqref{eq:hibler-strain-rate}. In addition, we will show in this paper, that a cutoff of the strain rate $\overline{\mathcal D}$ is not required in the dynamical case with equation \eqref{eq:evp-strain-rate} at hand for $ \mathcal E > 0 $. Note that in contrast to that of the Hibler model, the stress tensor of the EVP model has become an independent dynamical quantity, which requires its own initial condition which need to be chosen adequately in numerical simulations. Moreover, the parabolic structure of the full Hibler model \eqref{eq:hibler-momentum}-\eqref{eq:hibler-strain-rate} is lost due to this modification, although it reemerges as a damping term in the EVP model.

The EVP model, which is given by equations \eqref{eq:hibler-momentum} and \eqref{eq:evp-strain-rate}, is one of the most important dynamical sea-ice models and is a fundamental component of many climate models \cite{bouillon2009,golden,koldunov}, including the Los Alamos sea-ice model (CICE) \cite{lemieux2024}, the Louvain-la-Neuve (LIM2) sea-ice model \cite{bouillon2009}, the Massachusetts Institute of Technology general circulation model (MITgcm) \cite{losch2010}, and the Finite-Volume sea-ice–ocean model (FESOM) \cite{danilov} (an overview of several models can be found in \cite{hunke2020}). Furthermore, 
the wind stress $ \mathcal T_a $ and the ocean stress $ \mathcal T_w $ are defined as follows
\begin{align}
\label{def:windstress}
\mathcal{T}_{a} & \coloneqq c_a \rho_a \lvert U_a \rvert \bigg(U_a \cos \phi + U_a^\perp \sin \phi \bigg), \\
\label{def:oceanstress}
\mathcal{T}_{w} & \coloneqq c_w \rho_w \lvert U_w - u \rvert \bigg[ (U_w - u) \cos \theta + (U_w - u)^\perp \sin \theta \bigg],
\end{align}
\end{subequations}
respectively. 
The other notation is introduced in Table \ref{notationtable}. For the purposes of this paper, the geostrophic wind and ocean currents $U_a$ and $U_w$ can be arbitrary functions in $C([0,T]; H^4 (\mathbb{T}^2))$, but we note that the ocean velocity $U_w$ is typically chosen as follows, cf. e.g. \cite{hunkelinearization,danilov},
\begin{equation} \label{oceanvelocity}
U_w = \begin{pmatrix}
\frac{0.1 (2 y - L)}{L} \\
\frac{-0.1 (2x - L)}{L}
\end{pmatrix},
\end{equation}
where $L$ is the length scale of the domain $[0,L]^2$. In this paper we will take $L = 1$. Note that this example is not a spatially periodic vector field. 
Similarly, a typical wind velocity is given by (cf. \cite{hunkelinearization})
\begin{equation} \label{windvelocity}
U_a = \begin{pmatrix}
5 + \left[ \sin \left( \frac{2 \pi t}{T} \right) - 3 \right] \sin \left( \frac{2 \pi x}{L} \right) \sin \left( \frac{\pi y}{L} \right) \\
5 + \left[ \sin \left( \frac{2 \pi t}{T} \right) - 3 \right] \sin \left( \frac{2 \pi y}{L} \right) \sin \left( \frac{\pi x}{L} \right)
\end{pmatrix}.
\end{equation}

The goal of this paper is to establish the global well-posedness of a regularised version of the EVP system, i.e. equations \eqref{eq:hibler-momentum} and \eqref{eq:evp-strain-rate}-\eqref{def:oceanstress}. To simplify our presentation, we will fix $ \overline e \equiv 2 $ in the rest of this paper, as is commonly done in the sea-ice modelling literature. However, we emphasise that the results of this paper do not depend on the choice of $\overline{e}$.

\begin{table}[H]
    \centering
\begin{tabular}{ |p{2cm}||p{5cm}|p{3cm}|  }
\hline
Symbol & Meaning & Typical value\\
\hline
$c$ & ice compactness (area covered by `thick' ice) & $0 \leq c \leq 1$ \\
$c_a$ & air drag coefficient & $1.2 \cdot 10^{-3}$ \\
$c_w$ & ocean drag coefficient & $5.5 \cdot 10^{-3}$ \\
$\overline{e}$ & ratio of major axis to minor axis lengths of the elliptic yield curve & 2 \\
$\mathcal{E}$ & elastic modulus & $0.25$ \\
$g$ & gravitational constant & $9.81 \; \text{m\,s}^{-1}$ \\
$H_0$ & sea surface height & \\
$H$ & height of `thick` ice & \\
$m$ &  mass per unit area & \\
$\Omega$ &  rotation parameter & $1.46 \cdot 10^{-4} \; \text{s}^{-1}$ \\
$P$ & internal ice strength & \\
$P_0$ & internal ice strength parameter & $27.5 \cdot 10^3 \; \text{N/m}^3$ \\
$\phi$ & air turning angle & $25^\circ$ \\
$\rho_a$ & air density & $1.3 \; \text{kg/m}^3$ \\
$\rho_w$ & ocean water density & $1026 \; \text{kg/m}^3$ \\
$\theta$ & water turning angle & $25^\circ$ \\
$U_a$ & geostrophic wind & equation \eqref{windvelocity} \\
$U_w$ & geostrophic ocean current & equation \eqref{oceanvelocity} \\
 \hline
\end{tabular}
    \caption{Index of notation and typical values of several of the parameters, which are taken from \cite{hunke,hunkelinearization,mehlmann,lemieux}.}
    \label{notationtable}
\end{table}

\subsection{The linear ill-posedness of the unregularised EVP model} \label{illposednesssection}

Before we state the main result of this paper, we emphasise here that a regularisation of the EVP model is necessary to establish the well-posedness in Sobolev spaces. Indeed, consider the following one-dimensional version of the EVP model
\begin{subequations}\label{sys:1-d-evp}
\begin{align}
&\partial_t u  = \partial_x \sigma ,  \\
&\partial_t \sigma + \frac{5 \sqrt{\lvert \partial_x u \rvert^2 + \epsilon^2}}{2 P} \sigma  + \frac{\sqrt{\lvert \partial_x u \rvert^2 + \epsilon^2}}{2} = \partial_x u,
\end{align}
\end{subequations}
where $u$ and $\sigma$ are now scalar quantities, and we have regularised and simplified the strain rate $\overline{\mathcal{D}} = \sqrt{\frac{5}{4}} \lvert \partial_x u \rvert$ from \eqref{originalstrainrate} by $\sqrt{\lvert \partial_x u \rvert^2 + \epsilon^2}$ for some $\epsilon > 0$. Note that we have omitted the external forces as well as the advection term (cf. point $ (\romannumeral 2) $ of Remark \ref{rm:thm}). By linearising system \eqref{sys:1-d-evp} around a given solution $(\overline{u}, \overline{\sigma})$ to \eqref{sys:1-d-evp} we obtain the following system of equations
\begin{subequations}
\begin{align}
&\partial_t u_l = \partial_x \sigma_l \label{linearisedvelocityeq}, \\
&\partial_t \sigma_l = - \frac{5}{2P} \sqrt{\lvert \partial_x \overline{u}  \rvert^2 + \epsilon^2} \sigma_l + \bigg[ 1 - \frac{5}{2P} \frac{\overline{\sigma} \partial_x \overline{u}}{\sqrt{\lvert \partial_x \overline{u}  \rvert^2 + \epsilon^2}} - \frac{1}{2} \frac{\partial_x \overline{u} }{\sqrt{\lvert \partial_x \overline{u}  \rvert^2 + \epsilon^2}} \bigg] \partial_x u_l \label{linearisedstresseq},
\end{align}
\end{subequations}
where $(u_l,\sigma_l)$ is the solution of the linearised system \eqref{linearisedvelocityeq}-\eqref{linearisedstresseq}. Then, after taking a spatial derivative of equation \eqref{linearisedstresseq} and using equation \eqref{linearisedvelocityeq}, one has
\begin{equation}
\label{eq:linear-illposedness}
\begin{aligned}
\partial_{tt} u_l = &- \frac{5}{2P} \sqrt{\lvert \partial_x \overline{u}  \rvert^2 + \epsilon^2} \partial_t u_l + \bigg[ 1 - \frac{5}{2P} \frac{\overline{\sigma} \partial_x \overline{u}}{\sqrt{\lvert \partial_x \overline{u}  \rvert^2 + \epsilon^2}} - \frac{1}{2} \frac{\partial_x \overline{u} }{\sqrt{\lvert \partial_x \overline{u}  \rvert^2 + \epsilon^2}} \bigg] \partial_{xx} u_l \\
&- \frac{5}{2P} \frac{\partial_x \overline{u} \partial_{xx} \overline{u}}{\sqrt{\lvert \partial_x \overline{u}  \rvert^2 + \epsilon^2}} \sigma_l - \frac{\partial}{\partial x} \bigg[ \frac{5}{2P} \frac{\overline{\sigma} \partial_x \overline{u}}{\sqrt{\lvert \partial_x \overline{u}  \rvert^2 + \epsilon^2}} + \frac{1}{2} \frac{\partial_x \overline{u} }{\sqrt{\lvert \partial_x \overline{u}  \rvert^2 + \epsilon^2}} \bigg] \partial_x u_l.
\end{aligned}
\end{equation}
Therefore, if at a space-time point $(x_0, t_0) \in \mathbb{T} \times [0,T]$ we have
\begin{equation}
1 - \frac{5}{2P} \frac{\overline{\sigma} \partial_x \overline{u}}{\sqrt{\lvert \partial_x \overline{u}  \rvert^2 + \epsilon^2}} - \frac{1}{2} \frac{\partial_x \overline{u} }{\sqrt{\lvert \partial_x \overline{u}  \rvert^2 + \epsilon^2}} < 0,
\end{equation}
then equation \eqref{eq:linear-illposedness} becomes locally elliptic near the point $(x_0, t_0)$, which leads to an instability and therefore the ill-posedness in the sense of Hadamard \cite{hadamard}. In a forthcoming paper \cite{boutrosillposedness}, we will address the linear ill-posedness of the (unregularised) EVP model in all Sobolev spaces in a rigorous manner using this new instability (ill-posedness) that we have discussed above. 
\begin{remark}
We emphasise that the instability described above is fundamentally different than the (formal) ill-posedness/instability results that have been reported in the sea-ice modelling/physics literature \cite{pritchard,guba,gray1995,gray1999}. These results rely crucially on the absence of the viscosity cutoff (i.e. $\epsilon = 0$) and are found mainly for the Hibler model (with the exception of \cite{lipscomb,williams}, which consider the (revised) EVP model as well). The rigorous analysis in \cite{liu} demonstrates that the linear ill-posedness is not present in the case of the Hibler model with $\epsilon > 0$. The linear instability of the EVP model described in this section is present regardless of the value of $\epsilon$ and therefore pertains to a new kind of instability that has not been reported in the literature before, and in particular it does not apply to the Hibler model. 
\end{remark}

\subsection{The main result of this paper}
We consider the following regularisation of the EVP model (i.e. equations \eqref{eq:hibler-momentum} and \eqref{eq:evp-strain-rate}):
\begin{subequations}
\begin{gather}
\partial_t u  = \nabla \cdot \sigma + \mathcal{T}_a + \mathcal{T}_w + \Omega u^\perp - g \nabla H_0, \label{voigtregularisation1} \\
\dfrac{1}{\mathcal E} \partial_t \bigl( \sigma - \alpha^2 \Delta \sigma \bigr) + \dfrac{4 \overline{\mathcal D}}{P} (\sigma - \frac{1}{2} \Tr \sigma \mathbb I_2) + \dfrac{\overline{\mathcal D}}{2P} \Tr \sigma \mathbb I_2 + \dfrac{\overline{\mathcal D}}{2} \mathbb I_2  = D(u), \label{voigtregularisation2}
\\
u \lvert_{t = 0} = u_0, \quad \sigma \lvert_{t = 0} = \sigma_0. \label{voigtregularisation3}
\end{gather}
\end{subequations}
In this model we have included the Voigt regularisation on the elastic part of the rheology in the EVP model \cite{hunke} by means of the inclusion of the term $-\alpha^2 \partial_t \Delta \sigma$ for some fixed $\alpha > 0$ in equation \eqref{voigtregularisation2}. The drag forces $\mathcal{T}_a$ and $\mathcal{T}_w$ are taken as in equations \eqref{def:windstress} and \eqref{def:oceanstress}, respectively. Moreover, for the sake of simplicity we have chosen $m \equiv 1$. Note that throughout this paper we will assume that $\sigma (x,t)$ is a symmetric matrix, and in Lemma \ref{invariancelemma} we show rigorously that the stress tensor $\sigma$ remains symmetric if $\sigma_0$, i.e. initially, is symmetric. Next, we will simplify the strain rate $\overline{\mathcal{D}}$ from equation \eqref{originalstrainrate} by replacing it with
\begin{equation} 
\label{def:simplified_strain_rate}
\mathcal{D} \coloneqq \lvert D(u) \rvert,
\end{equation}
without fundamentally changing the analysis presented below, see Section \ref{conclusion} for more discussion. The system we will therefore consider in this paper, which we call the Voigt-EVP system, is given by
\begin{subequations}\label{sys:limitEVP}
\begin{gather}
\partial_t u  = \nabla \cdot \sigma + \mathcal{T}_a + \mathcal{T}_w + \Omega u^\perp - g \nabla H_0, \label{limitEVP1} \\
\dfrac{1}{\mathcal E} \partial_t \bigl( \sigma - \alpha^2 \Delta \sigma \bigr) + \dfrac{4 {\mathcal D}}{P} (\sigma - \frac{1}{2} \Tr \sigma \mathbb I_2) + \dfrac{{\mathcal D}}{2P} \Tr \sigma \mathbb I_2 + \dfrac{{\mathcal D}}{2} \mathbb I_2  = D(u),
\label{limitEVP2}  \\
u \lvert_{t = 0} = u_0, \quad \sigma \lvert_{t = 0} = \sigma_0. \label{limitEVP3}
\end{gather}
\end{subequations}

Now we are ready to state our main result, i.e., the global well-posedness of the Voigt-EVP system: 
\begin{theorem} \label{regularisedevplimit}
Let $u_0 \in H^2 (\mathbb{T}^2)$ and $\sigma_0 \in H^3 (\mathbb{T}^2)$ (which is symmetric) be given. Moreover, we assume that $ P > 0 $ is a given constant, $U_a, U_w, H_0 \in C([0,T]; H^4 (\mathbb{T}^2))$ and we let $\theta \in \left[0, \frac{\pi}{4} \right]$. The Voigt-EVP system \eqref{sys:limitEVP} has a unique global strong solution $(u,\sigma)$ on the time interval $[0,T]$ for any $T > 0$. The solution $(u,\sigma)$ enjoys the following regularity
\begin{equation*}
u \in C ([0,T]; H^2 (\mathbb{T}^2)), \quad \sigma \in C ([0,T]; H^3 (\mathbb{T}^2)),
\end{equation*}
and $\sigma$ is symmetric. In addition, we have the estimate
\begin{equation*}
\begin{gathered}
\lVert u \rVert_{C([0,T]; H^2 (\mathbb{T}^2))} + \lVert \sigma \rVert_{C([0,T]; H^3 (\mathbb{T}^2))} + \lVert \partial_t u \rVert_{L^\infty ((0,T); H^2 (\mathbb{T}^2))} + \lVert \partial_t \sigma \rVert_{L^\infty ((0,T); H^3 (\mathbb{T}^2))} \\ \leq C \exp (\exp (\exp (C T))),
\end{gathered}
\end{equation*}
where $C > 0$ is a constant, depending on $\mathcal E, \alpha, U_a, U_w, g, H_0, u_0, \sigma_0, c_a, c_w, \rho_a, \rho_w, P, \phi$, and $\theta$.
\end{theorem}
We note that the superexponential bound from Theorem \ref{regularisedevplimit} originates from the repeated use of (logarithmic) Grönwall inequalities for the a priori estimates for the $L^2 (\mathbb{T}^2)$, $\dot{H}^1 (\mathbb{T}^2)$ and $\dot{H}^2 (\mathbb{T}^2)$ norms.
\begin{remark}\label{rm:thm}
\begin{enumerate}[label=(\roman{enumi})]
    \item The Voigt regularisation in equation \eqref{limitEVP2} is in the same spirit as the original numerical regularisation of the Hibler model from \cite{hunke}. 
    Namely, we seek an inviscid regularisation of the EVP model, which is suitable for numerical computations. The advantage of the Voigt regularisation is that  
    on the one hand, like the original EVP model, it shares the same steady state(s) with the Hibler model. On the other hand, the Voigt regularisation constitutes a relatively modest (nonviscous) modification of the (artificial) elastic component of the rheology in the EVP model. 
    
    \item For the momentum equation of the Voigt-EVP model, we have omitted the advection term, in which we follow the approach in \cite{hunke}. As it has been remarked in \cite{hunke}, the advection term is at least one order of magnitude smaller than the acceleration term and therefore it has not been included in the EVP model. In the evolution equation for the stress tensor we also do not include an advection term, as it would affect the steady states of the EVP model, which would then lose its purpose as an approximation of the Hibler model. 

    \item We note that the assumption $\theta \in \left[0, \frac{\pi}{4} \right]$ on the water turning angle, lies within the standard range in sea-ice modelling (see for example \cite{heorton,lepparanta}). 

\end{enumerate}
\end{remark}

\smallskip

In order to prove Theorem \ref{regularisedevplimit}, we will consider an intermediate system given by
\begin{subequations}\label{sys:simplifiedmodel}
\begin{gather}
\partial_t u  = \nabla \cdot \sigma + \mathcal{T}_a + \mathcal{T}_w + \Omega u^\perp - g \nabla H_0, \label{simplifiedmodel1} \\
\dfrac{1}{\mathcal E} \partial_t \bigl( \sigma - \alpha^2 \Delta \sigma \bigr) + \dfrac{4 {\mathcal D_\epsilon}}{P} (\sigma - \frac{1}{2} \Tr \sigma \mathbb I_2) + \dfrac{{\mathcal D_\epsilon}}{2P} \Tr \sigma \mathbb I_2 + \dfrac{{\mathcal D_\epsilon}}{2} \mathbb I_2  = D(u),
\label{simplifiedmodel2} \\
u \lvert_{t = 0} = u_0, \quad \sigma \lvert_{t = 0} = \sigma_0. \label{simplifiedmodel3}
\end{gather}
\end{subequations}
In this model, the strain rate $\mathcal{D}$ in system \eqref{sys:limitEVP} has been regularised by
\begin{equation} \label{viscosityregularisation}
\mathcal{D}_\epsilon \coloneqq \sqrt{\lvert D(u) \rvert^2 + \epsilon^2 },
\end{equation}
for some $\epsilon > 0$, where we recall that $D(u)$ denotes the symmetric part of the gradient, see equation \eqref{def:symty-gradient}. 

Then, by establishing uniform-in-$\epsilon $ estimates of the solutions to system \eqref{sys:simplifiedmodel} globally in time, we will be able to send $ \epsilon \rightarrow 0 $ and obtain the (unique) solution to system \eqref{sys:limitEVP}. As a byproduct of this proof, we also obtain the global well-posedness of the intermediate system \eqref{simplifiedmodel1}-\eqref{simplifiedmodel3} as stated below:
\begin{corollary} \label{cor:intermediatesys}
Under the same conditions on $u_0, \sigma_0, U_a, U_w, H_0,$ and $\theta$ as in Theorem \ref{regularisedevplimit}, for any $T > 0$ the intermediate system \eqref{sys:simplifiedmodel} is globally well-posed. The unique global solution $(u,\sigma)$ has the same regularity as in Theorem \ref{regularisedevplimit}.
\end{corollary}

\begin{remark}\label{rk:epsilon}
\begin{enumerate}[label = (\roman{enumi})]
    \item The benefit of introducing $ \mathcal D_\epsilon $ in the intermediate system \eqref{sys:simplifiedmodel} is that it regularises the strain rate $ \mathcal D $ and allows one to take its derivatives directly. We emphasise that such a regularisation serves as a cutoff of the bulk and shear viscosities and has an important role in sea-ice modelling as it describes the yield curve of sea ice, as part of the viscous-plastic rheology introduced in \cite{hibler}. \\
    Moreover, it is a natural numerical regularisation for the strain rate, which was first introduced in \cite{kreyscher}, cf. \cite{mehlmann}. Note that in practice, $\epsilon$ should be smaller than $10^{-9} \;$s$^{-1}$ \cite{bouillon2009}. For these reasons the study of the intermediate system \eqref{sys:simplifiedmodel} has its value independent of its role as an approximate system for the Voigt-EVP model.
    \item The proof of Theorem \ref{regularisedevplimit} implies that the limit as $\epsilon \rightarrow 0$ of the intermediate system \eqref{sys:simplifiedmodel}, i.e. removing the bulk and shear viscosity cutoff, is stable in the sense that the solutions remain close to each other for all small $\epsilon$. Our results are therefore very much in agreement with the numerical simulations in \cite{bouillon2009}, where it was found that the dynamics of the EVP model is not very much affected by the value of $\epsilon$.
    \item In \cite{liu}, the approximation \eqref{viscosityregularisation} is also used for the Hibler model \eqref{eq:hibler-strain-rate} as a regularisation of the strain rate, and it is shown that for any fixed $ \epsilon > 0 $, the resulting system of equations 
    is locally well-posed. However, as demonstrated by the numerical evidence in \cite{bouillon2009}, the limit as $ \epsilon \rightarrow 0 $ for the Hibler model is very unstable. The analysis of the limit as $\epsilon \rightarrow 0$ remains open.

\item 
Instead of \eqref{viscosityregularisation}, our analysis works for a larger class of regularisations for the strain rate $ \mathcal D $. See Section \ref{conclusion} for more discussion.

\end{enumerate}
\end{remark}

\subsection{Overview of the literature}
\subsubsection*{The sea-ice modelling literature} Sea ice is a highly complex material and modelling its dynamics and rheology poses significant challenges. In fact, numerous sea-ice models have been developed over time \cite{golden}. The Arctic Ice Dynamics Joint Experiment suggested an elastic-plastic rheology for modelling sea ice \cite{coon}. The Hibler sea-ice model was introduced in \cite{hibler}, which uses a viscous-plastic rheology which has become one of the standard sea-ice rheologies. 

The Hibler model has been a very successful model in reproducing observed sea-ice drift \cite{losch}. However, its use in practical computations turns out to be very expensive \cite{bouillon2009,koldunov}, particularly so when using an explicit numerical method \cite{ip}. Instead, implicit methods are employed in order to compute solutions of the Hibler model, for example the line relaxation method \cite{zhang} as well as the Jacobian-free Newton Krylov solver \cite{lemieux,losch2014,seinen} to mention some.

The elastic-viscous-plastic (EVP) sea-ice model was introduced in \cite{hunke} in order to improve the computational efficiency of sea-ice modelling. Specifically, the EVP model is a modification of the Hibler model, as an elastic (relaxation) term is introduced into the constitutive relation. The advantage of the EVP formulation over the Hibler model is that it is expected to be more amenable to the use of explicit numerical schemes for computational modelling, and hence facilitates the use of parallel computing \cite{hunkelinearization}, see also \cite{hunke2002}. As a result, numerically solving the EVP model significantly reduces the computational cost compared to the use of implicit schemes for the Hibler model
\cite{bouillon2009}. 

In addition, the EVP model \eqref{eq:hibler-momentum} and \eqref{eq:evp-strain-rate} has the advantage that its computational performance is not affected by the value of the strain rate regularisation parameter $\epsilon$, while the Hibler model requires a small value of $\epsilon$ in the viscosity regularisation to accurately model plastic behaviour as shown in \cite{bouillon2009}, cf. Remark \ref{rk:epsilon}.
Because of the aforementioned reasons the EVP model has become one of the standard and most commonly used sea-ice models in contemporary climate modelling \cite{bouillon2009,golden,koldunov}.

Over time, several modifications of the EVP model were introduced which have better numerical performance in simulations. For example, in \cite{hunkelinearization} the elastic parameter was redefined to depend on the strain rate. The reason for this modification was to prevent the computed stresses from lying outside the yield curve (cf. \cite{hibler}) during the (numerical) subcycling process. 

\smallskip

However, although the EVP model was originally introduced as a reformulation of the Hibler model which was more suitable for numerical purposes, it remains unclear whether it successfully approximates the Hibler model, i.e. by sending $ \mathcal E \rightarrow \infty $ in equation \eqref{eq:evp-strain-rate} to obtain \eqref{eq:hibler-strain-rate}. In particular, it was found in \cite{lemieux,losch,losch2010,kimmritz}, see also \cite{bouillon}, that the numerically computed solutions of the EVP and the Hibler models can differ substantially, even for very small subcycling time steps. Note that this does not necessarily imply that either sea-ice model is physically more accurate, but just that their rheologies are fundamentally different from a computational point of view \cite{lemieux}.

Nevertheless, within the subcycling time steps, the effect of the artificial elasticity in the rheology is `damped' over the larger timescale, although it was found in \cite{lemieux} that the damping effect is weaker in areas of rigid ice, i.e. the regions with high concentration of ice and small strain rates. Finally, we note that several other sea-ice models have been developed in the literature. For example, in order to describe sea-ice dynamics across a range of different length scales, a multiscale model was recently developed in \cite{deng} (see also \cite{chen}).

\subsubsection*{The mathematical analysis literature}
The first results in the rigorous analysis of sea-ice dynamics only appeared very recently. The local well-posedness of the Hibler model was proved independently in \cite{liu,brandt}. Both of these papers slightly modified the original viscosity cutoff from \cite{hibler} by introducing a regularisation of the form \eqref{viscosityregularisation}, which is required to avoid the singularity when $D(u) = 0$. 

To be more precise, in \cite{brandt}, additional artificial diffusion terms were considered in the thermodynamics, which make the system fully parabolic. These additional diffusion terms made it possible to prove global existence of strong solutions for small data in \cite{brandt}. Meanwhile, the study in \cite{liu} retained the original mixed hyperbolic-parabolic character of the Hibler model. Later on, the same result as in \cite{liu} was also proved in \cite{brandt2025} with different techniques. A different regularisation of the viscosities was studied in \cite{chatta}, namely the replacement of the Heaviside cutoff by a hyperbolic tangent, where the authors establish the linear well-posedness (see also \cite{chatta2025} for a study of the two-dimensional case). 

The interaction between sea ice and a rigid body for the fully parabolic regularisation of the Hibler model was then studied in \cite{binz2024}. See \cite{binz2022} for some results on the coupling with oceanic and atmospheric dynamics. The free-boundary problem of shallow (land) ice-sheets was studied in \cite{piersanti}.

In addition, in \cite{mehlmann}, under the assumption of the existence of smooth solutions, a formal $H^1$-estimate was derived for the Hibler model as well as for a modified version of the EVP model in the formulation introduced in \cite{hunkelinearization}. This formulation of the EVP model from \cite{hunkelinearization} has a very different structure to our setting. 
The gradient estimates in \cite{mehlmann} strongly rely on the cutoff of the bulk and shear viscosities in the constitutive relation, and also on the linearisation of the ocean drag term. In the revised EVP case, the additional assumption that $\Tr \sigma = 0$ is made in \cite{mehlmann}. The a priori estimates in \cite{mehlmann} do not seem sufficient to deduce the existence of solutions for the revised EVP model, but instead are used in \cite{mehlmann} to check the consistency of a numerical discretisation scheme for these sea-ice models.

We also note that the EVP model \eqref{eq:hibler-momentum} and \eqref{eq:evp-strain-rate} shares some structural similarities with several models for non-Newtonian flows, for example the Oldroyd-B model, see for example \cite{guillope,lions,chemin,kupferman,lin,constantin2012,elgindi,constantin2021} and references therein, and other viscoelastic flow models, which have been studied in \cite{bulicek,bathory,bulicek2024} and references therein. Finally, we note that our treatment of the viscosity regularisation \eqref{viscosityregularisation} is in some sense related to results for non-Newtonian fluid models with shear dependent viscosities, see for example \cite{berselli} and references therein.

As has been mentioned in point $ (\romannumeral 1) $ in Remark \ref{rm:thm}, in this work we have chosen to use the (inviscid) Voigt regularisation for the EVP model. To the best of our knowledge, the Voigt regularisation in the context of the Navier-Stokes equations has first been rigorously studied in \cite{oskolkov1977,oskolkov1997}. Much later, the existence of solutions for such Kelvin-Voigt-type systems with non-constant density has also been investigated in \cite{antontsev} and see references therein. Moreover, the Voigt regularisation has been investigated extensively in the context of turbulence modelling, for example in \cite{larios2010,larios2010boussinesq,larios2014}. More recently, it has also been used as a regularisation mechanism to study steady states \cite{constantin2023}, see also \cite{ignatova}. We also remark that employing the Voigt regularisation typically leads to equations of pseudo-parabolic type, which have been studied in \cite{showalter1970a,showalter1970b,showalter1975} (and see references therein).

\smallskip

The rest of this paper is organised as follows. In Section \ref{preliminariessection}, we will recall some inequalities that we will use throughout this paper. In Section \ref{sec:intermediatesys}, we will establish the global well-posedness of the intermediate system \eqref{sys:simplifiedmodel}. Following this, in Section \ref{regularisedevpsection}, we finish the proof of Theorem \ref{regularisedevplimit}, namely the global well-posedness of the Voigt-EVP model \eqref{sys:limitEVP}. We will provide some concluding remarks in Section \ref{conclusion}.

\section{Preliminaries} \label{preliminariessection}

Throughout this paper, we will use the notation $ C $ to denote a generic constant, which might differ from line to line. For any quantities $ A, B \geq 0$, we shorten the notation $ A \leq C B $ for some positive but finite constant $C$ as
\begin{equation}
    A \lesssim B.
\end{equation}

We recall a few results that will be used in this paper. We first recall the Brézis-Gallouët–Wainger inequality, which was originally proved in \cite{brezis1979,brezis1980}:
\begin{lemma}[Brézis-Gallouët–Wainger inequality] \label{brezisgallouet}
Let $v \in H^2 (\mathbb{T}^2)$ (where $v$ is dimensionless, i.e. it does not have physical units), then it satisfies the following inequality
\begin{equation}\label{ineq:bgw}
\lVert v \rVert_{L^\infty} \leq C \left( \lVert v \rVert_{H^1} \sqrt{\log \bigg(1 + C_L \lVert v \rVert_{H^2} \bigg)} + 1 \right),
\end{equation}
where we note that the constant $C_L$ depends on the length scale of the torus (which we take to be unit length in our case), while $C$ is a dimensionless constant.
\end{lemma}
We will also use the Ladyzhenskaya inequality in two space dimensions. Let $f \in H^1 (\mathbb{T}^2)$, then we have
\begin{equation}\label{ineq:2d-ldzsky}
\lVert f \rVert_{L^4} \lesssim \lVert f \rVert_{L^2}^{1/2} \lVert \nabla f \rVert_{L^2}^{1/2} + \lVert f \rVert_{L^2}.
\end{equation}
In addition, we recall the logarithmic Gr\"onwall inequality, which in this form can be found in \cite[~Lemma 2.5]{cao2020}:
\begin{lemma}[Logarithmic Gr\"onwall inequality] \label{logarithmicgronwall}
Let $T > 0$, and assume $A : [0,T] \rightarrow [0,\infty)$ is a nonnegative absolutely continuous function. Further assume there exist nonnegative $L^1 (0,T)$ functions $f, g,$ and $h$ such that
\begin{equation}
\frac{\rm d}{\dt} A (t) \leq \big[ g(t) + h(t) \log (A (t) + e) \big] A(t) + f(t).
\end{equation}
Then we have the following estimate
\begin{equation}
A (t) \leq (2 F(t) + 1) e^{F(t)},
\end{equation}
where we have introduced the following function
\begin{equation}
F(t) \coloneqq e^{\int_0^t h(s) \ {\rm d} s} \bigg( \log (A(0) + e) + \int_0^t \big[ g(s) + f(s) \big] \ {\rm d} s + t \bigg).
\end{equation}
\end{lemma}

In order to prove the uniqueness of strong solutions in the subsequent sections, we will need to prove a bound on the difference of the strain rates.
\begin{lemma} \label{strainratelemma}
Let $v_1, v_2 \in H^2 (\mathbb{T}^2)$ and recall the notation $D(v)$ defined in equation \eqref{def:symty-gradient}, then we have
\begin{equation}\label{ineq:difference}
\left\lVert \sqrt{\lvert D(v_1) \rvert^2 + \epsilon^2} - \sqrt{\lvert D(v_2) \rvert^2 + \epsilon^2} \right\rVert_{L^2} \leq \left\lVert \nabla v_1 - \nabla v_2 \right\rVert_{L^2},
\end{equation}
for all $ \epsilon \geq 0 $. 
\end{lemma}

\begin{proof}
    One can directly compute
    \begin{gather*}
        \vert \sqrt{\lvert D(v_1) \rvert^2 + \epsilon^2} - \sqrt{\lvert D(v_2) \rvert^2 + \epsilon^2} \vert = \left\vert \dfrac{\lbrack D(v_1) + D(v_2) \rbrack: \lbrack D(v_1) - D(v_2) \rbrack}{\sqrt{\lvert D(v_1) \rvert^2 + \epsilon^2} + \sqrt{\lvert D(v_2) \rvert^2 + \epsilon^2}} \right\vert \\
        \leq \vert D(v_1) - D(v_2) \vert.
    \end{gather*}
    This implies inequality \eqref{ineq:difference} for smooth $ v_1 $ and $ v_2 $. For $ v_1, v_2 \in H^2(\mathbb T^2) $, one can proceed by choosing  smooth approximating sequences of $ H^2(\mathbb T^2) $ functions and conclude the proof. 
\end{proof}

We show in the following lemma that $ \sigma $ remains symmetric under the evolution of the Voigt-EVP model if it is initially symmetric:
\begin{lemma}[Invariance of the symmetry] \label{invariancelemma}
Consider a solution $ (u,\sigma) $ to either system \eqref{sys:limitEVP} or \eqref{sys:simplifiedmodel}, which has the regularity described in Theorem \ref{regularisedevplimit}, with $ \sigma_0 \in H^3 (\mathbb{T}^2; \mathbb{R}^{2 \times 2})$ being symmetric. That is, $ u \in C([0,T];H^2(\mathbb T^2)) $ and $ \sigma \in C([0,T];H^3(\mathbb T^2)) $
for arbitrary $ T > 0 $. 
Then $ \sigma(x,t) = \sigma(x,t)^\top $ for all $ 0 \leq t \leq T $ and all $x \in \mathbb{T}^2$. 
\end{lemma}

\begin{proof}
    Denote the antisymmetric part of $ \sigma $ by
    \begin{equation}
        W(\sigma) \coloneqq \frac{1}{2}(\sigma - \sigma^\top). 
    \end{equation}
    Then, from equation \eqref{limitEVP2} or equation \eqref{simplifiedmodel2}, one has 
    \begin{equation}
        \frac{1}{\mathcal{E}} \partial_t (I - \alpha^2 \Delta) W(\sigma) + \frac{4 H}{P} W( \sigma) = 0,
    \end{equation}
    for $ H = \mathcal D $ or $ H = \mathcal D_\epsilon $. Note that this equation holds in $C([0,T]; H^1 (\mathbb{T}^2))$. Subsequently, the standard $ L^2 $-estimate for $W(\sigma)$ yields 
    \begin{equation}
    \sup_{t \in [0,T]} \big[ \lVert W(\sigma) (t) \rVert_{L^2}^2 + \alpha^2 \lVert \nabla W(\sigma) (t) \rVert_{L^2}^2 \big] \leq \lVert W(\sigma_0) \rVert_{L^2}^2 + \alpha^2 \lVert \nabla W(\sigma_0) \rVert_{L^2}^2 = 0. 
    \end{equation}
    This finishes the proof. 
\end{proof}

\section{Global solutions to the intermediate system \eqref{sys:simplifiedmodel}}\label{sec:intermediatesys}

Our strategy of establishing the global solvability of system \eqref{sys:simplifiedmodel} is the following: 
First, we will write down the Galerkin approximation scheme. At each level of the scheme, by directly applying the Picard theory we obtain a local-in-time solution  of the Galerkin system. 
Second, we establish better regularity for the approximating solutions (as well as their global existence in time), and pass to the limit in the Galerkin scheme using the Aubin-Lions compactness theorem. Finally, we establish the uniqueness of the strong solution. 

\subsection{The Galerkin approximation scheme}
We look for a solution of the type
\begin{equation*}
u_N(x,t) = \sum_{k \in \mathbb{Z}^2, \lvert k \rvert \leq N} \widehat{u}_k(t) e^{2 \pi i k \cdot x}, \quad \sigma_N(x,t) = \sum_{k \in \mathbb{Z}^2, \lvert k \rvert \leq N} \widehat{\sigma}_k(t) e^{2 \pi i k \cdot x},
\end{equation*}
which satisfies the following Galerkin system (together with the usual condition that $u_N$ and $\sigma_N$ are real-valued):
\begin{subequations} \label{sys:galerkin}
\begin{gather}
\partial_t u_N  = \nabla \cdot \sigma_N + \mathbf{P}_N \mathcal{T}_a + \mathbf{P}_N \mathcal{T}_w^N + \Omega u^\perp_N - g \mathbf{P}_N \nabla H_0, \label{galerkin1} \\
\begin{gathered} 
\frac{1}{\mathcal{E}} \partial_t \big( \sigma_N - \alpha^2 \Delta \sigma_N \big) + \mathbf{Q}_N \bigg[\frac{4 \mathcal{D}_\epsilon^N}{P} (\sigma_N - \frac{1}{2} \Tr \sigma_N \mathbb I_2) \bigg] + \mathbf{Q}_N \bigg[ \frac{\mathcal{D}_\epsilon^N}{2 P} \Tr \sigma_N \mathbb I_2 \bigg] \\ 
+ \frac{\mathbf{Q}_N \mathcal{D}_\epsilon^N}{2} \mathbb I_2 = D(u_N), 
\end{gathered}\label{galerkin2} \\
u \lvert_{t = 0} = \mathbf{P}_N u_0, \quad \sigma \lvert_{t = 0} = \mathbf{Q}_N \sigma_0, \label{galerkin3}
\end{gather}
\end{subequations}
where $\mathbf{P}_N$ is the projection of two-dimensional vector fields onto the Fourier modes up to order $N$, including the zeroth mode, while $\mathbf{Q}_N$ is the projection of $2 \times 2$ (symmetric) matrices onto the Fourier modes up to order $N$.  
In system \eqref{sys:galerkin}, we have used the notation
\begin{align}
\label{def:galerkin-oceanstress}
\mathcal{T}_{w}^N &= c_w \rho_w \lvert U_w - u_N \rvert \bigg[ (U_w - u_N) \cos \theta + (U_w - u_N)^\perp \sin \theta \bigg] \\
\intertext{and}
\label{def:galerkin-strainrate}
\mathcal{D}_\epsilon^N &= \sqrt{\lvert D(u_N) \rvert^2 + \epsilon^2 }.
\end{align}
Notice that system \eqref{sys:galerkin} is a first-order system of ordinary differential equations for $ \lbrace \hat u_k, \hat\sigma_k \rbrace_{\vert k \vert \leq N} $, with a locally Lipschitz continuous vector field with respect to the unknowns. 
Thus, the Picard theory (see for example \cite{evans}) implies that, for each initial datum $(\mathbf{P}_N u_0, \mathbf{Q}_N \sigma_0)$, there exists a unique local-in-time solution. 

\subsection{Uniform-in-$ (N, \epsilon) $ estimates}
To establish the uniform-in-$ N $ estimates for the solutions of system \eqref{sys:galerkin}, let
\begin{equation}\label{def:sigma_p}
    \tau_N \coloneqq \sigma_N + \frac{P}{2} \mathbb I_2.
\end{equation}
Recall that the pressure $ P > 0 $ is constant. We then rewrite equation \eqref{galerkin2} as follows:
\begin{equation}
     \begin{gathered}
     \frac{1}{\mathcal{E}} \partial_t \bigg( \tau_N  - \alpha^2 \Delta \tau_N \bigg) + \mathbf{Q}_N \bigg[\frac{4 \mathcal{D}_\epsilon^N}{P} \bigg( \tau_N - \frac{1}{2} \Tr \tau_N  \mathbb I_2 \bigg) \bigg]  \\
     + \mathbf{Q}_N \bigg[ \frac{\mathcal{D}_\epsilon^N}{2P} \Tr \tau_N  \mathbb I_2 \bigg] = D(u_N).
     \end{gathered}\label{rearrangedequation}
\end{equation}

{\noindent\bf $ L^2 $-estimate.} We start by computing the $L^2 (\mathbb{T}^2)$ estimates for the solution of system \eqref{sys:galerkin} on its maximal (positive) time interval of existence. Taking the $L^2 $-inner product of \eqref{galerkin1} and \eqref{rearrangedequation} with $u_N$ and $ \tau_N $, respectively, applying integration by parts, and summing the results yields
\begin{equation}\label{est:galerkin-l2}
\begin{gathered}
\frac{1}{2} \frac{\rm d}{\dt} \bigg\lbrack \lVert u_N \rVert_{L^2}^2 + \mathcal E^{-1} \Vert \tau_N \Vert_{L^2}^2 + \alpha^2 \mathcal E^{-1} \Vert \nabla \tau_N \Vert_{L^2}^2 \biggr\rbrack \\
+ \int_{\mathbb T^2} \biggl\lbrack \frac{4 \mathcal D_\epsilon^N}{P} \vert \tau_N - \frac{1}{2}\Tr \tau_N \mathbb I_2 \vert^2 + \frac{\mathcal D_\epsilon^N}{2P} \vert \Tr \tau_N\vert^2 \biggr\rbrack \,\dx \\
= \int_{\mathbb{T}^2} \bigg[\mathcal{T}_a + \mathcal{T}_w^N + \Omega u^\perp_N - g \nabla H_0 \bigg] \cdot u_N \dx \leq C + C \Vert u_N \Vert_{L^2}^2, 
\end{gathered}
\end{equation}
thanks to the following estimate:
\begin{gather*}
\int_{\mathbb{T}^2} \mathcal{T}_w^N \cdot u_N \dx = c_w \rho_w \int_{\mathbb{T}^2} \lvert U_w - u_N \rvert \bigg[ (U_w - u_N) \cos \theta + (U_w - u_N)^\perp \sin \theta \bigg] \cdot u_N \dx \\
= c_w \rho_w \int_{\mathbb{T}^2} \lvert U_w - u_N \rvert \bigg[ U_w \cos \theta + U_w^\perp \sin \theta \bigg] \cdot u_N \dx \\ - c_w \rho_w \int_{\mathbb{T}^2} \lvert U_w - u_N \rvert \lvert u_N \rvert^2 \cos \theta \dx 
\lesssim 1 + \lVert u_N \rVert_{L^2}^2, 
\end{gather*}
where we use the fact that  $\cos \theta \geq 0$ since $\theta \in [0, \frac{\pi}{4}]$.
Notice that the constant $ C \in (0,\infty) $ in \eqref{est:galerkin-l2} depends only on $\Omega, g, P, H_0, U_a, U_w, c_a, c_w, \phi, \theta, \rho_a,$ and $\rho_w$, but not on $\epsilon$ and $N$. We remark that estimate \eqref{est:galerkin-l2} would also hold if one considers the intermediate model with the regularised and nonsimplified strain rate $\overline{\mathcal{D}}_\epsilon^N$ instead of $\mathcal{D}_\epsilon^N$ (as both have definite sign). Moreover, we also have used the following cancellation, which is similar in structure to the Oldroyd-B model,
\begin{equation} \label{eq:cancellation}
\int_{\mathbb{T}^2} \bigg[ u_N \cdot (\nabla \cdot \sigma_N) + \tau_N : D(u_N) \bigg] \dx = 0.
\end{equation}
Therefore, applying the Gr\"onwall inequality to estimate \eqref{est:galerkin-l2} on a time interval of existence $[0,T]$ of the approximating Galerkin solution leads to
\begin{equation}\label{L2bound}
\sup_{t \in [0,T]} \bigg[ \lVert u_N (t) \rVert_{L^2}^2 + \mathcal E^{-1} \left\lVert \sigma_N (t) + \frac{P}{2} \mathbb I_2 \right\rVert_{L^2}^2 + \alpha^2 \mathcal E^{-1} \left\lVert \nabla \bigg[ \sigma_N (t) + \frac{P}{2} \mathbb I_2 \bigg] \right\rVert_{L^2}^2 \bigg] \leq C e^{C T},
\end{equation}
uniformly in $N$ (and also $\epsilon$). Since the right-hand side remains finite for any $T \in (0,\infty)$, we conclude that the time interval of existence of the approximating Galerkin solutions is $[0,\infty)$ and estimate \eqref{L2bound} therefore holds for any $T \in (0,\infty)$.

\smallskip 
{\noindent\bf ${H}^1$-estimate.}
Now we turn to the $\dot{H}^1 (\mathbb{T}^2)$ estimates. Denote by 
\begin{equation}\label{def:relative-velocity}
    V_w := U_w - u_N. 
\end{equation}
Recalling \eqref{def:galerkin-oceanstress}, one can write  
\begin{equation}\label{est:h1-oceanstress}
\begin{gathered}
- \int_{\mathbb T^2} \mathbf{P}_N \mathcal T_w^N \cdot \Delta u_N \dx = \sum_{i=1}^2 \int_{\mathbb{T}^2} \mathbf{P}_N \partial_i \mathcal{T}_w^N \cdot \partial_i u_N \dx = \sum_{i=1}^2  \int_{\mathbb{T}^2} \partial_i \mathcal{T}_w^N \cdot \partial_i u_N \dx \\
 = \sum_{i=1}^2 \int_{\mathbb T^2} \partial_i \mathcal T_w^N \cdot \partial_i {U_w} \dx - \sum_{i=1}^2 \int_{\mathbb T^2} c_w \rho_w \bigl\lbrack  \partial_i (\lvert V_w \rvert V_w ) \cos\theta \bigr\rbrack \cdot \partial _i V_w \dx\\
 - \sum_{i=1}^2 \int_{\mathbb{T}^2} c_w \rho_w \partial_i(\vert V_w \vert V_w^\perp) \cdot \partial_i V_w \sin \theta \dx =: I_1 + I_2 + I_3,
\end{gathered}
\end{equation}
where we used that the Galerkin projection commutes with derivatives. For $ i =1,2 $, one can calculate
\begin{equation} \label{firstderivativeoceanstress}
\partial_i(\vert V_w \vert V_w ) = \vert V_w \vert \partial_i V_w + \dfrac{V_w \cdot \partial_i V_w}{\vert V_w \vert} V_w.
\end{equation}
Then directly one can estimate (where the constants can depend on $\lVert \partial_i {U_w} \rVert_{L^\infty}$)
\begin{gather}
    \label{est:I-1}
    \vert I_1 \vert \lesssim \Vert V_w \Vert_{L^2} \Vert \nabla V_w \Vert_{L^2} \lesssim 1 + \Vert{u_N}\Vert_{L^2}^2 + \Vert{\nabla u_N}\Vert_{L^2}^2,\\
    \label{est:I-3}
    \vert I_3 \vert \leq \left\lvert \sum_{i=1}^2 \int_{\mathbb{T}^2} c_w \rho_w \left( \vert V_w \vert \partial_i V_w^\perp + \dfrac{V_w \cdot \partial_i V_w}{\vert V_w \vert} V_w^\perp \right) \cdot \partial_i V_w \sin \theta \dx \right\rvert  \nonumber \\
    = \left\lvert \sum_{i=1}^2 \int_{\mathbb{T}^2} c_w \rho_w \dfrac{V_w \cdot \partial_i V_w}{\vert V_w \vert} V_w^\perp \cdot \partial_i V_w \sin \theta \dx \right\rvert \\
    \leq c_w \rho_w \vert \sin \theta \vert \int_{\mathbb T^2} \vert V_w \vert \vert \nabla V_w \vert^2 \dx.
\end{gather}
Meanwhile, substituting \eqref{firstderivativeoceanstress} in $ I_2 $, one can compute
\begin{equation}\label{est:I-2}
        I_2 = - c_w \rho_w \cos \theta \int_{\mathbb T^2} \left(\vert V_w \vert \vert\nabla V_w \vert^2 + \frac{1}{4} \frac{\left\vert \nabla \big( \vert V_w \vert^2 \big) \right\vert^2}{\vert V_w \vert} \right) \dx.
\end{equation}
Consequently, \eqref{est:h1-oceanstress}--\eqref{est:I-2} leads to
\begin{equation}
\label{est:h1-oceanstress-1}
    - \int_{\mathbb T^2} \mathbf{P}_N \mathcal T_w^N \cdot \Delta u_N \dx = \int_{\mathbb T^2} \nabla \mathcal T_w^N : \nabla u_N \lesssim 1 + \Vert{u_N}\Vert_{L^2}^2 + \Vert{\nabla u_N}\Vert_{L^2}^2 
\end{equation}
for $  \theta  $ satisfying 
\begin{equation}
\label{theta-cndt}
    \cos \theta - \vert \sin \theta \vert \geq 0,
\end{equation}
which holds for $\theta \in \left[0, \frac{\pi}{4} \right]$.
\smallskip 

Taking the $ L^2 $-inner product of \eqref{galerkin1} and \eqref{galerkin2} with $ - \Delta u_N$ and $- \Delta \sigma_N$, respectively, applying integration by parts, and summing up the results leads to the following estimate
\begin{align}
&\frac{1}{2} \frac{\rm d}{\dt} \bigg[ \lVert \nabla u_N \rVert_{L^2}^2 + \mathcal{E}^{-1} \lVert \nabla \sigma_N \rVert_{L^2}^2 + \alpha^2 \mathcal{E}^{-1} \lVert \Delta \sigma_N \rVert_{L^2}^2 \bigg] \nonumber \\
= &\int_{\mathbb{T}^2} \biggl\lbrace \biggl\lbrack \frac{4 \mathcal D_\epsilon^N}{P}(\sigma_N - \frac{1}{2} \Tr\sigma_N \mathbb I_2) \biggr\rbrack : \Delta \sigma_N + \frac{\mathcal D_\epsilon^N}{2P} \Tr\sigma_N \mathbb I_2 : \Delta \sigma_N + \frac{\mathcal D_\epsilon^N}{2} \Tr \Delta \sigma_N \biggr\rbrace \dx 
\nonumber \\ &+ \int_{\mathbb{T}^2} \nabla \bigg[\mathcal{T}_a + \Omega u^\perp_N - g \nabla H_0 \bigg] : \nabla u_N \dx + \int_{\mathbb T^2} \nabla \mathcal{T}_w^N :\nabla u_N \dx  \nonumber \\
&+ \sum_{i=1}^2 \int_{\mathbb{T}^2} \bigg[ \partial_i u_N \cdot (\nabla \cdot \partial_i \sigma_N) + \partial_i \sigma_N : D(\partial_i u_N) \bigg] dx \nonumber \\
\lesssim \, &\lVert \mathcal{D}_\epsilon^N \rVert_{L^2} \lVert \sigma_N \rVert_{L^\infty} \lVert \Delta \sigma_N \rVert_{L^2} + \lVert \Delta \sigma_N \rVert_{L^2}^2 + \lVert \nabla u_N \rVert_{L^2}^2 + \Vert u_N \Vert_{L^2}^2 + 1 \nonumber \\
\lesssim \, &\left[ \sqrt{\log (e + C_L R_1 \lVert \sigma_N \rVert_{H^2})} \lVert \sigma_N \rVert_{H^1} + 1 \right] \big[ \lVert \nabla u_N \rVert_{L^2}^2 + \lVert \Delta \sigma_N \rVert_{L^2}^2 \big] \nonumber \\
&+ \lVert \Delta \sigma_N \rVert_{L^2}^2 + \lVert \nabla u_N \rVert_{L^2}^2 + \Vert u_N \Vert_{L^2}^2 + 1, \nonumber \\
\lesssim \, &\left[ \log (e + C_L R_1 \lVert \sigma_N \rVert_{H^2}) + \lVert \sigma_N \rVert_{H^1}^2 + 1 \right] \big[ \lVert \nabla u_N \rVert_{L^2}^2 + \lVert \Delta \sigma_N \rVert_{L^2}^2 \big] \nonumber \\
&+ \lVert \Delta \sigma_N \rVert_{L^2}^2 + \lVert \nabla u_N \rVert_{L^2}^2 + \Vert u_N \Vert_{L^2}^2 + 1 \label{est:H-1}
\end{align}
where we have used \eqref{est:h1-oceanstress-1}, the Br\'ezis-Gallou\"et-Wainger inequality \eqref{ineq:bgw}, as well as Young's inequality. Note that we have introduced a constant $R_1 > 0$ in order to make the quantity $R_1 \sigma_N$ dimensionless (so that we can apply the Br\'ezis-Gallou\"et-Wainger inequality), and in particular $R_1 \sigma_N$ does not have physical units. One can choose $R_1 = \lVert \sigma_0 \rVert_{L^\infty}^{-1}$ for example (for nonzero initial data). We observe that for the ${H}^1$-estimate one can also treat the nonsimplified strain rate $\overline{\mathcal{D}}_\epsilon^N$ instead of $\mathcal{D}_\epsilon^N$ by using estimate \eqref{originalstrainrateestimate}. In addition, we have used the following cancellation identity
\begin{equation}
\sum_{i=1}^2 \int_{\mathbb{T}^2} \bigg[ \partial_i u_N \cdot (\nabla \cdot \partial_i \sigma_N) + \partial_i \sigma_N : D(\partial_i u_N) \bigg] \ {\rm d} x = 0.
\end{equation}

Thanks to the $L^2$-estimate in \eqref{L2bound}, we can apply the logarithmic Gr\"onwall inequality from Lemma \ref{logarithmicgronwall} to \eqref{est:H-1}, where we set 
\begin{equation*}
A(t) \coloneqq \lVert u_N \rVert_{H^1}^2 + \lVert \sigma_N \rVert_{H^2}^2, \quad f(t) \coloneqq 1, \quad g(t) \coloneqq 1 + \lVert \sigma_N \rVert_{H^1}^2, \quad h (t) \coloneqq 1.
\end{equation*}
This then leads to the following double-exponential bound
\begin{equation} \label{H1bound}
\sup_{t \in [0,T]} \biggl\lbrack \Vert \nabla u_N \Vert_{L^2}^2 + \mathcal E^{-1} \Vert \nabla \sigma_N \Vert_{L^2}^2  + \alpha^2 \mathcal E^{-1} \Vert \Delta \sigma_N \Vert_{L^2}^2 \biggr\rbrack
\leq C \exp (\exp (C T)),
\end{equation}
for some $ C $ independent of $N, \epsilon $, and $ T $, but depending on $ \mathcal E, \alpha $, and the initial data. Note that the double-exponential nature of the bound comes from the application of the logarithmic Gr\"onwall inequality to the single-exponential $L^2$-estimate \eqref{L2bound}.

\smallskip 
{\noindent\bf $ H^2 $-estimate.} 
One can compute directly that
\begin{equation}\label{est:dd-sigma}
\begin{aligned}
\nabla \mathcal{D}_\epsilon^N &= \frac{D(u_N) : \big[ \nabla D(u_N) \big]}{\mathcal{D}_\epsilon^N} = \frac{\frac{1}{2} \nabla \lvert D (u_N) \rvert^2}{\mathcal{D}_\epsilon^N}, \\
\lvert \nabla \mathcal{D}_\epsilon^N \rvert  &\leq \frac{\lvert D(u_N) \rvert \lvert \nabla D(u_N) \rvert }{\mathcal{D}_\epsilon^N} = \frac{\lvert D(u_N) \rvert}{\sqrt{\lvert D(u_N) \rvert^2 + \epsilon^2 }} \lvert \nabla D(u_N) \rvert \leq \lvert \nabla D(u_N) \rvert,  \\
\lVert \nabla \mathcal{D}_\epsilon^N \rVert_{L^2} &\lesssim \lVert \Delta u_N \rVert_{L^2}, \qquad \text{where the constant is independent of $ \epsilon $ and $ N $}.
\end{aligned}
\end{equation}
In addition, we recall the following computation from \cite[~p. 22-23]{liu}, for $i,j=1,2$, with $ V_w $ given in \eqref{def:relative-velocity},
\begin{equation}
    \begin{aligned}
         \partial^2_{ij} \big( \lvert V_w \rvert V_w \big) = & \; \lvert V_w  \rvert \partial^2_{ij} V_w + \frac{V_w  \cdot \partial_i V_w }{\lvert V_w \rvert} \partial_j V_w + \frac{V_w \cdot \partial_j V_w}{\lvert V_w  \rvert} \partial_i V_w  \\
        &+ \bigg( \frac{V_w  \cdot \partial^2_{ij} V_w + \partial_i V_w  \cdot \partial_j V_w }{\lvert V_w \rvert} - \frac{\big[ V_w  \cdot \partial_i V_w  \big] \cdot \big[ V_w  \cdot \partial_j V_w  \big]}{\lvert V_w \rvert^3} \bigg) V_w .
    \end{aligned}
\end{equation}
Therefore, one has the following estimate
\begin{align}
\big\lVert \nabla^2 \big( \lvert V_w \rvert V_w \big) \big\rVert_{L^2} &\leq \lVert V_w \rVert_{L^\infty} \lVert V_w \rVert_{H^2} + \lVert V_w \rVert_{W^{1,4}}^2 \nonumber \\
&\lesssim \big( \lVert V_w \rVert_{L^\infty} + \lVert  V_w \rVert_{H^1} \big) \lVert V_w \rVert_{H^2}, \label{est:dd-os}
\end{align}
where we have applied the Ladyzhenskaya inequality \eqref{ineq:2d-ldzsky}. 

Then, after taking the $ L^2 $-inner product of \eqref{galerkin1} and \eqref{galerkin2} with $ \Delta^2 u_N $ and $ \Delta^2 \sigma_N $, respectively, using integration by parts, and summing up the results, one obtains
\begin{align} \label{est:H-2}
&\frac{1}{2} \frac{\rm d}{\dt} \bigg[ \lVert \Delta u_N \rVert_{L^2}^2 + \mathcal{E}^{-1} \lVert \Delta \sigma_N \rVert_{L^2}^2 + \alpha^2 \mathcal{E}^{-1} \lVert \nabla \Delta \sigma_N \rVert_{L^2}^2 \bigg] \nonumber \\
=  &- \int_{\mathbb T^2}\bigg[\frac{4 \mathcal D_\epsilon^N}{P}\vert \Delta(\sigma_N - \frac{1}{2}\Tr \sigma_N \mathbb I_2)\vert^2 + \frac{\mathcal D_\epsilon^N}{2P} \vert \Delta\Tr \sigma_N \vert^2 \bigg] \dx \nonumber \\
 &- \sum_{i=1,2}\int_{\mathbb T^2} \biggl\lbrack \frac{4 \partial_i \mathcal D_\epsilon^N}{P}\partial_i(\sigma_N - \frac{1}{2}\Tr \sigma_N \mathbb I_2) + \frac{\partial_i \mathcal D_\epsilon^N}{2P} \partial_i \Tr \sigma_N \mathbb I_2  \biggr\rbrack : \Delta \sigma_N \dx \nonumber \\
 &+ \sum_{i=1,2}\int_{\mathbb T^2} \biggl\lbrack \frac{4 \partial_i \mathcal D_\epsilon^N}{P}(\sigma_N - \frac{1}{2}\Tr \sigma_N \mathbb I_2) + \frac{\partial_i \mathcal D_\epsilon^N}{2P}  \Tr \sigma_N \mathbb I_2  \biggr\rbrack : \partial_i \Delta \sigma_N \dx \nonumber \\ 
 &+ \int_{\mathbb{T}^2} \nabla \left( \frac{\mathcal D_\epsilon^N}{2} \mathbb I_2 \right) : \nabla \Delta \sigma_N \dx 
 + \int_{\mathbb{T}^2} \Delta \bigg[\mathcal{T}_a + \mathcal{T}_w^N + \Omega u^\perp_N - g \nabla H_0 \bigg] \cdot \Delta u_N \dx \nonumber \\
 &+ \int_{\mathbb{T}^2} \bigg[ \Delta u_N \cdot (\nabla \cdot \Delta \sigma_N) + \Delta \sigma_N : D(\Delta u_N) \bigg] \ {\rm d} x \nonumber \\
\lesssim \, &\lVert \nabla^2 u_N \rVert_{L^2} \lVert \nabla \sigma_N \rVert_{L^4} \lVert \Delta \sigma_N \rVert_{L^4} + \big[ \lVert \nabla^2 u_N \rVert_{L^2} \lVert \sigma_N \rVert_{L^\infty} +  \lVert \nabla^2 u_N \rVert_{L^2} \big] \lVert \nabla \Delta \sigma_N \rVert_{L^2}  \nonumber \\
&+ \big( 1 + \lVert u_N \rVert_{L^\infty} + \lVert \nabla u_N \rVert_{L^2} \big) \bigl( 1+ \lVert u_N \rVert_{H^2} \bigr) \lVert u_N \rVert_{H^2} 
+ \lVert \Delta u_N \rVert_{L^2}^2  + 1 \nonumber \\
\lesssim &\bigg[ \lVert \sigma_N \rVert_{H^1} \big( 1 + \sqrt{\log \big( e + C_L R_1 \lVert \sigma_N \rVert_{H^2} \big)} \big) + \lVert u_N \rVert_{H^1} \big( 1 + \sqrt{\log \big( e + C_L R_2 \lVert u_N \rVert_{H^2} \big)} \big)  \nonumber \\
 &+ \Vert \sigma_N \Vert_{H^2} + 1 \bigg] 
\cdot \big[ \lVert u_N \rVert_{H^2}^2 + \lVert \sigma_N \rVert_{H^3}^2 \big] + 1 \nonumber \\
\lesssim &\bigg[ \lVert \sigma_N \rVert_{H^1}^2 + \log \big( e + C_L R_1 \lVert \sigma_N \rVert_{H^2} \big) + \lVert u_N \rVert_{H^1}^2 + \log \big( e + C_L R_2 \lVert u_N \rVert_{H^2} \big)   \nonumber \\
 &+ \Vert \sigma_N \Vert_{H^2} + 1 \bigg] 
\cdot \big[ \lVert u_N \rVert_{H^2}^2 + \lVert \sigma_N \rVert_{H^3}^2 \big] + 1,
\end{align}
where we have used \eqref{est:dd-os}, the Br\'ezis-Gallou\"et-Wainger inequality \eqref{ineq:bgw}, the Sobolev embedding theorem as well as Young's inequality. The terms on the second line are negative definite, and in the seventh line we estimate the term $\lVert \nabla^2 u_N \rVert_{L^2} \lVert \nabla \sigma_N \rVert_{L^4} \lVert \Delta \sigma_N \rVert_{L^4}$ by means of the Sobolev embedding theorem, which gives $\lVert u_N \rVert_{H^2} \lVert \sigma_N \rVert_{H^2} \lVert \sigma_N \rVert_{H^3}$. Applying Young's inequality leads to some of the terms on the last line of estimate \eqref{est:H-2}. Due to the monotonicity of the logarithm, we can replace the arguments $\lVert \sigma_N \rVert_{H^2}$ and $\lVert u_N \rVert_{H^2}$ by their sum, which facilitates the application of the logarithmic Grönwall inequality. 

Note that as before, we have used parameters $R_1, R_2 > 0$ in order to make the quantities $R_1 \sigma_N$ and $R_2 u_N$ dimensionless (such that the resulting quantities do not have physical units), so that we can apply the Br\'ezis-Gallou\"et-Wainger inequality. One can choose $R_2 = \lVert u_0 \rVert_{L^\infty}^{-1}$ for example (in the case of nonzero initial data). Moreover, similarly to before we have used the following cancellation
\begin{equation}
\int_{\mathbb{T}^2} \bigg[ \Delta u_N \cdot (\nabla \cdot \Delta \sigma_N) + \Delta \sigma_N : D(\Delta u_N) \bigg] \ {\rm d} x = 0.
\end{equation}

From \eqref{est:H-2}, by using the $L^2$-estimate \eqref{L2bound}, the $ H^1$-estimate \eqref{H1bound}, and the logarithmic Gr\"onwall inequality in Lemma \ref{logarithmicgronwall}, we find the following bound:
\begin{equation} \label{galerkinbound1}
\sup_{t \in [0,T]} \biggl\lbrack \lVert u_N (t) \rVert_{H^2}^2 + \lVert \sigma_N (t) \rVert_{H^3}^2 \biggr\rbrack \leq C \exp (\exp (\exp (C T))),
\end{equation}
for some constant $ C $ which is independent of $ N, \epsilon $, and $ T $, but which depends on $ \mathcal E, \alpha $, and the initial data. Note that the triple-exponential growth in time of the bound originates from the single-exponential respectively double-exponential growth in time of the $L^2$ and $H^1$-estimates, and the subsequent application of the logarithmic Gr\"onwall inequality.

\smallskip 

{\noindent\bf Estimate on the time derivative and summary.}
Thanks to \eqref{galerkinbound1} and system \eqref{sys:galerkin}, one can obtain directly that
\begin{equation}
    \label{galerkinbound-2}
    \sup_{t \in [0,T]} \biggl\lbrack \Vert \partial_t u_N(t) \Vert_{H^2}^2 + \Vert \partial_t \sigma_N - \alpha^2 \partial_t\Delta \sigma_N \Vert_{H^1}^2 \biggr\rbrack \leq C \exp (\exp (\exp (C T))).
\end{equation}
This estimate follows because all the other terms (i.e. except the time derivative) in equation \eqref{galerkin1} can be shown to lie in $C([0,T];H^2(\mathbb{T}^2))$, while all the other terms in equation \eqref{galerkin2} are elements of $C([0,T];H^1(\mathbb{T}^2))$ (because of estimate \eqref{galerkinbound1}). 
Meanwhile, by directly applying integration by parts we deduce that
\begin{gather*}
    \Vert \partial_t \sigma_N - \alpha^2 \partial_t \Delta \sigma_N \Vert_{H^1}^2 = \Vert \partial_t \sigma_N \Vert^2_{H^1} + \alpha^4 \Vert \partial_t \Delta \sigma_N \Vert^2_{H^1} + 2 \alpha^2 \Vert \partial_t \nabla \sigma_N \Vert_{H^1}^2.
\end{gather*}
Hence, equations \eqref{galerkinbound1} and \eqref{galerkinbound-2} yield
\begin{equation}\label{galerkinbound-total}
    \begin{gathered}
        \sup_{t \in [0,T]} \biggl\lbrack \lVert u_N (t) \rVert_{H^2} + \lVert \sigma_N (t) \rVert_{H^3} + 
        \Vert \partial_t u_N(t) \Vert_{H^2} + \Vert   \partial_t\sigma_N \Vert_{H^3}
        \biggr\rbrack\\
        \leq C \exp (\exp (\exp (C T))),
    \end{gathered}
\end{equation}
again for some constant $ C $, which is independent of $ N, \epsilon $, and $ T $, but depends on $ \mathcal E, \alpha $, and the initial data.

\subsection{Well-posedness of the intermediate system \eqref{sys:simplifiedmodel}}\label{sec:wel-posedness-intermediate}

{\noindent\bf Passing the limit $ N \rightarrow \infty $ and existence of the strong solution.}
For any fixed $ T \in (0,\infty) $ and $ \epsilon \in (0,\infty) $, from the bound \eqref{galerkinbound-total} and by applying the Aubin-Lions compactness lemma (see \cite[Theorem II.5.16]{boyer}), one has the following convergence results: for a suitable subsequence of $ N \rightarrow \infty $, 
\begin{align}
    \label{weak-u-N} u_N & \overset{\ast}{\rightharpoonup} u &  \text{weakly-$* $ in}& ~ L^\infty ((0,T);H^2 (\mathbb{T}^2)), \\
    \label{weak-dt-u-N} \partial_t u_N & \overset{\ast}{\rightharpoonup} \partial_t u &  \text{weakly-$* $ in}& ~ L^\infty ((0,T);H^2 (\mathbb{T}^2)),  \\
    \label{strong-n-N} u_N & \rightarrow u & \text{strongly in}& ~ C([0,T];H^1(\mathbb T^2)), \\
    \label{weak-sgm-N} \sigma_N &\overset{\ast}{\rightharpoonup} \sigma & \text{weakly-$* $ in}& ~ L^\infty ((0,T);H^3 (\mathbb{T}^2)), \\
    \label{strong-sgm-N} \sigma_N &\rightarrow \sigma & \text{strongly in}& ~ C([0,T];H^2(\mathbb T^2)),\\
    \label{weak-dt-sgm-N} \partial_t \sigma_N & \overset{\ast}{\rightharpoonup} \partial_t \sigma & \text{weakly-$* $ in}& ~ L^\infty ((0,T);H^3 (\mathbb{T}^2)),
\end{align}
where $ u \in L^\infty((0,T);H^2(\mathbb T^2))$, $ \partial_t u \in L^\infty((0,T);H^2(\mathbb T^2)) $, $ \sigma \in L^\infty((0,T);H^3(\mathbb T^2)) $, and $ \partial_t \sigma \in L^\infty((0,T);H^3(\mathbb T^2)) $ with the same estimates as in \eqref{galerkinbound-total} with $ (u_N,\sigma_N) $ replaced by $ (u,\sigma) $, i.e.,
\begin{equation}\label{intermediate-total}
    \begin{gathered}
        \sup_{t \in [0,T]} \biggl\lbrack \lVert u (t) \rVert_{H^2} + \lVert \sigma (t) \rVert_{H^3} + 
        \Vert \partial_t u(t) \Vert_{H^2} + \Vert   \partial_t\sigma \Vert_{H^3}
        \biggr\rbrack\\
        \leq C \exp (\exp (\exp (C T))),
    \end{gathered}
\end{equation}
for some constant $ C $ independent of $ \epsilon $ and $ T $, but depending on $ \mathcal E, \alpha $, and the initial data. 
We also have $u \in C([0,T]; H^2 (\mathbb{T}^2))$ and $\sigma \in C([0,T]; H^3 (\mathbb{T}^2))$ thanks to \cite[Lemma 1.3.2, p. 196]{sohr}.

It is clear that one can pass to the limit in the linear terms of the Galerkin system \eqref{sys:galerkin}. To show the convergence of the nonlinear terms, first we observe that $\mathcal{D}_\epsilon^N \rightarrow \mathcal{D}_\epsilon$ in $C ([0,T]; L^2 (\mathbb{T}^2))$. In fact, due to Lemma \ref{strainratelemma}, we have by using \eqref{strong-n-N}
\begin{equation*}
\lVert \mathcal{D}_\epsilon^N - \mathcal{D}_\epsilon \rVert_{C([0,T]; L^2(\mathbb{T}^2))} \leq C \lVert u_N - u \rVert_{C([0,T]; H^1 (\mathbb{T}^2))} \xrightarrow[]{N \rightarrow \infty} 0.
\end{equation*}
Notice that this bound holds pointwise in time, as will generally be the case for the estimates in this part of the paper. We will therefore omit the temporal part of the norm in this section. Next, by using \eqref{strong-sgm-N} one can estimate that
\begin{align*}
&\bigg\lVert \frac{4 \mathcal{D}_\epsilon^N}{P} \sigma_N - \frac{4}{P} \mathcal{D}_\epsilon \sigma \bigg\rVert_{L^2} \lesssim \lVert \mathcal{D}_\epsilon^N - \mathcal{D}_\epsilon \rVert_{L^2} \lVert \sigma_N \rVert_{L^\infty} + \lVert \mathcal{D}_\epsilon \rVert_{L^2} \lVert \sigma_N - \sigma \rVert_{L^\infty} \xrightarrow[]{N \rightarrow \infty} 0,
\end{align*}
by the convergence of both $\mathcal{D}_\epsilon^N$ as well as $\sigma_N$. Similarly, we know that the other nonlinear terms in the evolution equation for the stress tensor, i.e. $-\frac{3}{2 P} \mathcal{D}_\epsilon^N \Tr \sigma_N I$ and $\frac{\mathcal{D}_\epsilon^N}{2} I$, converge to $-\frac{3}{2 P} \mathcal{D}_\epsilon \Tr \sigma I$ and $\frac{\mathcal{D}_\epsilon}{2} I$ in $L^\infty ((0,T); L^2 (\mathbb{T}^2))$ respectively. 

Finally, we need to show that the ocean stress term $\mathcal{T}_w^N$ converges. We have
\begin{align*}
\lVert \mathcal{T}_{w}^N - \mathcal{T}_{w} \rVert_{L^2} &\lesssim \lVert \lvert U_w - u \rvert - \lvert U_w - u_N \rvert \lVert_{L^4} \left\lVert \bigg[ (U_w - u_N) \cos \theta + (U_w - u_N)^\perp \sin \theta \bigg] \right\rVert_{L^4} \\
&+ \lVert \lvert U_w - u \rvert \rVert_{L^4} \left\lVert \bigg[ (u - u_N) \cos \theta + (u - u_N)^\perp \sin \theta \bigg] \right\rVert_{L^4} \\
&\lesssim \lVert u - u_N \rVert_{L^4} \xrightarrow[]{N \rightarrow \infty} 0.
\end{align*}
Therefore all the terms in the Galerkin system \eqref{sys:galerkin} converge, and $ (u,\sigma ) $ solves the intermediate system \eqref{sys:simplifiedmodel}. 
This concludes the proof of the existence of a strong solution to the intermediate model \eqref{sys:simplifiedmodel}. 

\smallskip 
{\noindent\bf Stability and uniqueness.}

Suppose that $(u_1, \sigma_1), (u_2, \sigma_2) \in C ([0,T]; H^2 (\mathbb{T}^2)) \times C ([0,T]; H^3 (\mathbb{T}^2))$ are two strong solutions to the intermediate model \eqref{sys:simplifiedmodel}. Let
\begin{equation*}
\delta u \coloneqq u_1 - u_2, \quad \delta \sigma \coloneqq \sigma_1 - \sigma_2. 
\end{equation*}
The system satisfied by the difference $(\delta u, \delta \sigma)$ is then given by
\begin{subequations}\label{sys:sb-uq}
\begin{align}
\label{eq:sb-uq-1}
\partial_t \delta u  = \nabla \cdot \delta \sigma + \mathcal{T}_{w,1} - \mathcal{T}_{w,2} + \Omega (\delta u)^\perp ,  \\
\label{eq:sb-uq-2}
\frac{1}{\mathcal{E}} \partial_t \big( \delta \sigma - \alpha^2 \Delta \delta \sigma \big) + \frac{4 \mathcal D_{\epsilon,1}}{P}(\sigma_1 - \frac{1}{2} \Tr \sigma_1 \mathbb I_2) - \frac{4 \mathcal D_{\epsilon,2}}{P}(\sigma_2 - \frac{1}{2} \Tr \sigma_2 &\mathbb I_2 ) \\
+ \frac{\mathcal D_{\epsilon,1}}{2P} \Tr \sigma_1 \mathbb I_2 - \frac{\mathcal D_{\epsilon,2}}{2P} \Tr \sigma_2 \mathbb I_2 
+ \frac{\mathcal{D}_{\epsilon,1} - \mathcal{D}_{\epsilon,2}}{2} &\mathbb I_2 = D(\delta u),
\end{align}
where we have introduced the following notation for $i=1,2$
\begin{align}
\mathcal{T}_{w,i} &= c_w \rho_w \lvert U_w - u_i \rvert \big[ (U_w - u_i) \cos \theta + (U_w - u_i)^\perp \sin \theta \big], \\
\mathcal{D}_{\epsilon,i} &= \sqrt{\lvert D(u_i) \rvert^2 + \epsilon^2 }.
\end{align} 
\end{subequations}
We observe that $\partial_t (\delta u) \in L^\infty ((0,T); H^2 (\mathbb{T}^2))$ and $\partial_t (\delta \sigma) \in L^\infty ((0,T); H^3 (\mathbb{T}^2))$, and therefore we can take the $L^2$-inner product of equations \eqref{eq:sb-uq-1} and \eqref{eq:sb-uq-2} with $\delta u$ respectively $\delta \sigma$. We therefore compute the $L^2$-estimates on the difference $(\delta u, \delta \sigma)$ as follows (using a cancellation similar to \eqref{eq:cancellation})
\begin{equation}\label{est:L-2-diff-1}
\begin{gathered}
\frac{1}{2} \frac{\rm d}{\dt} \bigg[ \lVert \delta u \rVert_{L^2}^2 + \mathcal{E}^{-1} \lVert \delta \sigma \rVert_{L^2}^2 + \mathcal{E}^{-1} \alpha^2 \lVert \nabla \delta \sigma \rVert_{L^2}^2 \bigg] = 
- \int_{\mathbb{T}^2} \frac{\mathcal{D}_{\epsilon,1} - \mathcal{D}_{\epsilon,2}}{2} \Tr \delta \sigma \dx \\
 - \int_{\mathbb{T}^2} \bigg[ \frac{4 \mathcal D_{\epsilon,1}}{P}(\sigma_1 - \frac{1}{2} \Tr \sigma_1 \mathbb I_2) - \frac{4 \mathcal D_{\epsilon,2}}{P}(\sigma_2 - \frac{1}{2} \Tr \sigma_2 \mathbb I_2) \bigg] : \delta \sigma \dx \\
 - \int_{\mathbb T^2} \bigg[ \frac{\mathcal D_{\epsilon,1}}{2P} \Tr \sigma_1 \mathbb I_2 - \frac{\mathcal D_{\epsilon,2}}{2P} \Tr \sigma_2 \mathbb I_2 \bigg] : \delta \sigma \dx \\
+ \int_{\mathbb{T}^2} \big( \mathcal{T}_{w,1} - \mathcal{T}_{w,2} \big) \cdot \delta u \dx  =: J_1 + J_2 + J_3 + J_4.
\end{gathered}
\end{equation}
Applying H\"older's inequality, inequality \eqref{ineq:difference}, and the Sobolev embedding theorem yields
\begin{align*}
    \vert J_3 \vert & \lesssim \int_{\mathbb T^2} \vert \mathcal D_{\epsilon,1} - \mathcal D_{\epsilon,2} \vert \lvert \Tr \sigma_1 \vert \vert \Tr \delta \sigma \vert + \vert \mathcal D_{\epsilon,2} \vert \vert \Tr \delta \sigma \vert^2 \dx \\
    & \lesssim \lVert \sigma_1 \rVert_{L^\infty} \lVert \mathcal{D}_{\epsilon,1} - \mathcal{D}_{\epsilon,2} \rVert_{L^2} \lVert \delta \sigma \rVert_{L^2} + \lVert \mathcal{D}_{\epsilon,2} \rVert_{L^2} \lVert \delta \sigma \rVert_{L^4}^2 \\
    & \lesssim \lVert \sigma_1 \rVert_{L^\infty} \big[ \lVert \nabla \delta u \rVert_{L^2}^2 + \lVert \delta \sigma \rVert_{L^2}^2 \big] + \lVert \nabla u_2 \rVert_{L^2} \big[ \lVert \delta \sigma \rVert_{L^2}^2 + \lVert \nabla \delta \sigma \rVert_{L^2}^2 \big].
\end{align*}
Similarly, we find
\begin{align*}
\vert J_2 \vert 
& \lesssim \lVert \sigma_1 \rVert_{L^\infty} \big[ \lVert \nabla \delta u \rVert_{L^2}^2 + \lVert \delta \sigma \rVert_{L^2}^2 \big] + \lVert \nabla u_2 \rVert_{L^2} \big[ \lVert \delta \sigma \rVert_{L^2}^2 + \lVert \nabla \delta \sigma \rVert_{L^2}^2 \big],\\
 \vert J_1 \vert 
& \lesssim \lVert \nabla \delta u \rVert_{L^2}  \lVert \delta \sigma \rVert_{L^2},
\intertext{and}
\vert J_4 \vert & \lesssim (\lVert u_1 \rVert_{L^\infty} + \lVert u_2 \rVert_{L^\infty} + \lVert U_w \rVert_{L^\infty}) \lVert \delta u \rVert_{L^2}^2.
\end{align*}
Combining these bounds we get that
\begin{equation}\label{est:L-2-diff}
\begin{gathered}
    \frac{1}{2} \frac{\rm d}{\dt} \bigg[ \lVert \delta u \rVert_{L^2}^2 + \mathcal{E}^{-1} \lVert \delta \sigma \rVert_{L^2}^2 + \mathcal{E}^{-1} \alpha^2 \lVert \nabla \delta \sigma \rVert_{L^2}^2 \bigg] \lesssim (\lVert \sigma_1 \rVert_{L^\infty} + 1) \big[ \lVert \nabla \delta u \rVert_{L^2}^2 + \lVert \delta \sigma \rVert_{L^2}^2 \big] \\
    + \lVert \nabla u_2 \rVert_{L^2} \big[ \lVert \delta \sigma \rVert_{L^2}^2 + \lVert \nabla \delta \sigma \rVert_{L^2}^2 \big] + (\lVert u_1 \rVert_{L^\infty} + \lVert u_2 \rVert_{L\infty} + \lVert U_w \rVert_{L^\infty}) \lVert \delta u \rVert_{L^2}^2.
\end{gathered}
\end{equation}

\smallskip 
Meanwhile, to calculate the $ H^1 $ estimate of the difference $ (\delta u , \delta \sigma ) $, similar to \eqref{est:H-1}, we have (by using a cancellation similar to \eqref{eq:cancellation})
\begin{equation} \label{est:H-1-diff-1}
\begin{aligned}
&\frac{1}{2} \frac{\rm d}{\dt} \bigg[ \lVert \nabla \delta u \rVert_{L^2}^2 + \mathcal{E}^{-1} \lVert \nabla \delta \sigma \rVert_{L^2}^2 + \alpha^2 \mathcal{E}^{-1} \lVert \Delta \delta \sigma \rVert_{L^2}^2 \bigg] \\
= &\int_{\mathbb{T}^2} \frac{\mathcal{D}_{\epsilon,1} - \mathcal{D}_{\epsilon,2}}{2} \Tr \Delta \delta \sigma \dx + \int_{\mathbb{T}^2} \bigg[ \frac{4 \mathcal D_{\epsilon,1}}{P}(\sigma_1 - \frac{1}{2} \Tr \sigma_1 \mathbb I_2) - \frac{4 \mathcal D_{\epsilon,2}}{P}(\sigma_2 - \frac{1}{2} \Tr \sigma_2 \mathbb I_2) \bigg] : \Delta \delta \sigma \dx \\
 &+ \int_{\mathbb T^2} \bigg[ \frac{\mathcal D_{\epsilon,1}}{2P} \Tr \sigma_1 \mathbb I_2 - \frac{\mathcal D_{\epsilon,2}}{2P} \Tr \sigma_2 \mathbb I_2 \bigg] : \Delta \delta \sigma \dx + \int_{\mathbb{T}^2} \nabla\big( \mathcal{T}_{w,1} - \mathcal{T}_{w,2} \big) : \nabla \delta u \dx \\
\eqqcolon & \; J_5 + J_6 + J_7 + J_8 .
\end{aligned}\end{equation}
Similarly to the $L^2$-estimate, we establish (by using inequality \eqref{ineq:difference} and the Sobolev embedding theorem)
\begin{align*}
\vert J_5 + J_6 + J_7 \vert  
&\lesssim \lVert \sigma_1 \rVert_{L^\infty} \lVert \mathcal{D}_{\epsilon,1} - \mathcal{D}_{\epsilon,2} \rVert_{L^2} \lVert \Delta \delta \sigma \rVert_{L^2} + \lVert \mathcal{D}_{\epsilon,2} \rVert_{L^2} \lVert \delta \sigma \rVert_{L^\infty} \lVert \Delta \delta \sigma \rVert_{L^2} \\
& \qquad \qquad + \lVert \mathcal{D}_{\epsilon,1} - \mathcal{D}_{\epsilon,2} \rVert_{L^2} \lVert \Delta \delta \sigma \rVert_{L^2} \\
&\lesssim \big( \lVert \sigma_1 \rVert_{L^\infty} + \lVert \nabla u_2 \rVert_{L^2} + 1 \big) \big[ \lVert \nabla \delta u \rVert_{L^2}^2 + \lVert \delta \sigma \rVert_{L^2}^2 + \lVert \Delta \delta \sigma \rVert_{L^2}^2 \big].
\end{align*} 
Finally, we estimate $ J_8 $. Recalling the notation \eqref{def:relative-velocity}, we denote by
\begin{equation}\label{def:rlv-i}
    V_{w,i} \coloneqq U_w - u_i, \qquad i = 1,2,
\end{equation}
and hence $ V_{w,2} - V_{w,1} = \delta u $. 
Thus, by using identity \eqref{firstderivativeoceanstress}, one can compute, for $i=1,2$,
\begin{align*}
\partial_i (\vert V_{w,1} \vert V_{w,1} - \vert V_{w,2} \vert V_{w,2}) = & - \vert V_{w,1} \vert \partial_i \delta u + \biggl[ \vert V_{w,1} \vert - \vert V_{w,2} \vert \biggr] \partial_i V_{w,2}  \\
&- \frac{V_{w,1} \cdot \partial_i V_{w,1}}{\vert V_{w,1} \vert} \delta u - \frac{V_{w,1} \cdot \partial_i \delta u}{\vert V_{w,1} \vert} V_{w,2} - \frac{\delta u \cdot \partial_i V_{w,2}}{\vert V_{w,2} \vert} V_{w,2} \\
&- \bigg[\vert V_{w,1} \vert - \vert V_{w,2} \vert \biggr] \frac{V_{w,1} \cdot \partial_i V_{w,2}}{\vert V_{w,1} \vert \vert V_{w,2}\vert} V_{w,2}.
\end{align*} 
Note that in the above computation we have used the following relation
\begin{align*}
\frac{V_{w,1} \cdot \partial_i V_{w,1}}{\vert V_{w,1} \vert} V_{w,1} &- \frac{V_{w,2} \cdot \partial_i V_{w,2}}{\vert V_{w,2} \vert} V_{w,2} = \frac{V_{w,1} \cdot \partial_i V_{w,1}}{\vert V_{w,1} \vert} (V_{w,1} - V_{w,2}) + \frac{V_{w,1} \cdot \partial_i (V_{w,1} - V_{w,2})}{\vert V_{w,1} \vert} V_{w,2} \\
&+ \bigg(\frac{1}{\vert V_{w,1} \vert} - \frac{1}{\vert V_{w,2} \vert} \bigg) (V_{w,1} \cdot \partial_i V_{w,2}) V_{w,2} + \frac{(V_{w,1} - V_{w,2}) \cdot \partial_i V_{w,2}}{\vert V_{w,2} \vert} V_{w,2}.
\end{align*}
Therefore, one has that, for $i=1,2$, 
\begin{equation}\label{est:diff-tsnoc-1}
    \Vert \partial_i (\mathcal T_{w,1} - \mathcal T_{w,2} ) \Vert_{L^2} \lesssim (\Vert V_{w,1} \Vert_{L^\infty} + \Vert V_{w,2} \Vert_{L^\infty}) \Vert \partial_i \delta u \Vert_{L^2} + (\Vert \partial_i V_{w,1} \Vert_{L^4} + \Vert \partial_i V_{w,2} \Vert_{L^4}) \Vert \delta u \Vert_{L^4}.
\end{equation}
Hence one has by using \eqref{def:rlv-i}
\begin{equation}\label{est:diff-tsnc-2}
\begin{aligned}
\vert J_8 \vert 
&\lesssim \lVert \nabla \delta u \rVert_{L^2} \bigg[  \big( \lVert U_w - u_1 \rVert_{L^\infty} + \lVert U_w - u_2 \rVert_{L^\infty} \big) \lVert \nabla \delta u \rVert_{L^2}\\
& \qquad \qquad +  \big( \lVert \nabla (U_w - u_1) \rVert_{L^4} + \lVert \nabla (U_w - u_2) \rVert_{L^4} \big) \lVert \delta u \rVert_{L^4} \bigg] \\
&\lesssim \big( \lVert U_w - u_1 \rVert_{H^2} + \lVert U_w - u_2 \rVert_{H^2} \big) \lVert \delta u \rVert_{H^1}^2.
\end{aligned}
\end{equation}
Thus, we have derived the following from estimate \eqref{est:H-1-diff-1}, 
\begin{equation}
    \label{est:H-1-diff}
    \begin{gathered}
        \frac{1}{2} \frac{\rm d}{\dt} \bigg[ \lVert \nabla \delta u \rVert_{L^2}^2 + \mathcal{E}^{-1} \lVert \nabla \delta \sigma \rVert_{L^2}^2 + \alpha^2 \mathcal{E}^{-1} \lVert \Delta \delta \sigma \rVert_{L^2}^2 \bigg] \\
        \lesssim (1 + \Vert \sigma_1 \Vert_{H^2} + \Vert \sigma_2 \Vert_{H^2} + \Vert u_1 \Vert_{H^2} + \Vert u_2 \Vert_{H^2}) (\Vert \delta u\Vert_{H^1}^2 + \Vert \delta \sigma\Vert_ {H^2}^2 ).
    \end{gathered}
\end{equation}

\bigskip 

Consequently, combining estimates \eqref{est:L-2-diff} and \eqref{est:H-1-diff} leads to 
\begin{equation}\label{est:diff-total}
\begin{gathered}
\frac{1}{2} \frac{\rm d}{\dt} \bigg[ \lVert \delta u \rVert_{L^2}^2 + \lVert \nabla \delta u \rVert_{L^2}^2 + \mathcal{E}^{-1} \lVert \delta \sigma \rVert_{L^2}^2 + \mathcal{E}^{-1} (1 + \alpha^2) \lVert \nabla \delta \sigma \rVert_{L^2}^2 + \mathcal{E}^{-1} \alpha^2 \lVert \Delta \delta \sigma \rVert_{L^2}^2 \bigg] \\
\lesssim (1 + \Vert \sigma_1 \Vert_{H^2} + \Vert \sigma_2 \Vert_{H^2} + \Vert u_1 \Vert_{H^2} + \Vert u_2 \Vert_{H^2}) (\Vert \delta u\Vert_{H^1}^2 + \Vert \delta \sigma\Vert_ {H^2}^2 ),
\end{gathered} \end{equation}
from which the stability (i.e. the continuous dependence on the initial data) and uniqueness of the strong solution follows. This concludes the proof of the global well-posedness of the intermediate model \eqref{sys:simplifiedmodel}.

\section{Proof of Theorem \ref{regularisedevplimit}} \label{regularisedevpsection}

We consider a sequence of solutions $ \{ (u_\epsilon, \sigma_\epsilon) \}$ of the intermediate system \eqref{sys:simplifiedmodel} such that $\epsilon \rightarrow 0$. Thanks to estimate \eqref{intermediate-total}, which is independent of $ \epsilon $, by the Aubin-Lions lemma and Banach-Alaoglu theorem we have the following convergence results
\begin{align}
    u_\epsilon & \overset{\ast}{\rightharpoonup} u &  \text{weakly-$* $ in}& ~ L^\infty ((0,T);H^2 (\mathbb{T}^2)), \\
    \partial_t u_\epsilon & \overset{\ast}{\rightharpoonup} \partial_t u &  \text{weakly-$* $ in}& ~ L^\infty ((0,T);H^2 (\mathbb{T}^2)),  \\
    u_\epsilon & \rightarrow u & \text{strongly in}& ~ C([0,T];H^1(\mathbb T^2)), \\
    \sigma_\epsilon &\overset{\ast}{\rightharpoonup} \sigma & \text{weakly-$* $ in}& ~ L^\infty ((0,T);H^3 (\mathbb{T}^2)), \\
    \sigma_\epsilon &\rightarrow \sigma & \text{strongly in}& ~ C([0,T];H^2(\mathbb T^2)),\\
    \partial_t \sigma_\epsilon & \overset{\ast}{\rightharpoonup} \partial_t \sigma & \text{weakly-$* $ in}& ~ L^\infty ((0,T);H^3 (\mathbb{T}^2)),
\end{align}
where the limits $ u \in L^\infty((0,T);H^2(\mathbb T^2))$, $ \partial_t u \in L^\infty((0,T);H^2(\mathbb T^2)) $, $ \sigma \in L^\infty((0,T);H^3(\mathbb T^2)) $, and $ \partial_t \sigma \in L^\infty((0,T);H^3(\mathbb T^2)) $ satisfy estimate \eqref{intermediate-total}. It is left to verify that $ (u,\sigma) $ is a solution to the Voigt-EVP model \eqref{sys:limitEVP}.

First we prove that $\sqrt{\lvert D(u_\epsilon) \rvert^2 + \epsilon^2} \rightarrow \lvert D(u) \rvert$ in $L^\infty ((0,T); L^2 (\mathbb{T}^2))$. Indeed, note that for all $t \in [0,T]$
\begin{align*}
&\left\lVert \sqrt{\lvert D(u_\epsilon) \rvert^2 + \epsilon^2} - \lvert D(u) \rvert \right\rVert_{L^2} \leq \left\lVert \sqrt{\lvert D(u_\epsilon) \rvert^2 + \epsilon^2} - \sqrt{\lvert D(u) \rvert^2 + \epsilon^2} \right\rVert_{L^2} \\
&+ \left\lVert \sqrt{\lvert D(u) \rvert^2 + \epsilon^2} - \lvert D(u) \rvert \right\rVert_{L^2} \leq C \lVert u_\epsilon - u \rVert_{H^1} + \left\lVert \epsilon \right\rVert_{L^2} \xrightarrow[]{\epsilon \rightarrow 0} 0,
\end{align*}
where we have used Lemma \ref{strainratelemma}, the convergence $u_\epsilon \rightarrow u$ in $C([0,T];H^1(\mathbb{T}^2))$ and the elementary inequality $\sqrt{y + z} \leq \sqrt{y} + \sqrt{z}$ for $y, z \geq 0$. It is clear that the linear terms in equations \eqref{simplifiedmodel1}-\eqref{simplifiedmodel2} converge. We now show that the nonlinear terms in the equation for the stress tensor converge. In particular, for all $t \in [0,T]$ we have
\begin{align*}
\bigg\lVert \frac{4 \sqrt{\lvert D(u_\epsilon) \rvert^2 + \epsilon^2}}{P} \sigma_\epsilon - \frac{4}{P} \lvert D(u) \rvert \sigma \bigg\rVert_{L^2} &\lesssim \lVert \sqrt{\lvert D(u_\epsilon) \rvert^2 + \epsilon^2} - \lvert D(u) \rvert \rVert_{L^2} \lVert \sigma_\epsilon \rVert_{L^\infty} \\
&+ \lVert \lvert D(u) \rvert \rVert_{L^2} \lVert \sigma_\epsilon - \sigma \rVert_{L^\infty} \xrightarrow[]{\epsilon \rightarrow 0} 0.
\end{align*}
In a similar fashion, one can show that the terms $-\frac{3}{2} \frac{\sqrt{\lvert D(u_\epsilon) \rvert^2 + \epsilon^2}}{P}\Tr \sigma_\epsilon \mathbb I_2$ and $\frac{\sqrt{\lvert D(u_\epsilon) \rvert^2 + \epsilon^2}}{2} \mathbb I_2$ converge to $-\frac{3}{2} \frac{\lvert D(u) \rvert}{P}\Tr \sigma \mathbb I_2$ and $\frac{\lvert D(u) \rvert}{2} \mathbb I_2$ in $L^\infty ((0,T); L^2 (\mathbb{T}^2))$. Finally, we show that the ocean stress term $\mathcal{T}_w^\epsilon$ converges. We find
\begin{align*}
\lVert \mathcal{T}_{w,\epsilon} - \mathcal{T}_{w} \rVert_{L^2} \lesssim \; &\lVert \lvert U_w - u \rvert - \lvert U_w - u_\epsilon \rvert \lVert_{L^4} \left\lVert \bigg[ (U_w - u_\epsilon) \cos \theta + (U_w - u_\epsilon)^\perp \sin \theta \bigg] \right\rVert_{L^4} \\
&+ \lVert \lvert U_w - u \rvert \rVert_{L^4} \left\lVert \bigg[ (u - u_\epsilon) \cos \theta + (u - u_\epsilon)^\perp \sin \theta \bigg] \right\rVert_{L^4} \\
\lesssim \; &\lVert u - u_\epsilon \rVert_{L^4} \xrightarrow[]{\epsilon \rightarrow 0} 0,
\end{align*}
where again the convergence is uniform in time for the fixed time interval $[0,T]$. Therefore we conclude that $(u,\sigma)$ is a global strong solution of the Voigt-EVP model.

\smallskip 

To show the (global) well-posedness, one can repeat similar arguments as in Section \ref{sec:wel-posedness-intermediate}. If the Voigt-EVP model would have two strong solutions $(u_1, \sigma_1), (u_2, \sigma_2) \in C ([0,T]; H^2 (\mathbb{T}^2)) \times C ([0,T]; H^3 (\mathbb{T}^2))$, then we again consider the difference
\begin{equation*}
\delta u \coloneqq u_1 - u_2, \quad \delta \sigma \coloneqq \sigma_1 - \sigma_2. 
\end{equation*}
Similarly to before, one can check that the difference $(\delta u, \delta \sigma)$ satisfies the following system
\begin{subequations}
\begin{align}
\partial_t \delta u  = \nabla \cdot \delta \sigma + \mathcal{T}_{w,1} - \mathcal{T}_{w,2} + \Omega (\delta u)^\perp ,  \\
\frac{1}{\mathcal{E}} \partial_t \big( \delta \sigma - \alpha^2 \Delta \delta \sigma \big) + \frac{4 \mathcal D_1}{P}(\sigma_1 - \frac{1}{2} \Tr \sigma_1 \mathbb I_2) - \frac{4 \mathcal D_2}{P}(\sigma_2 - \frac{1}{2} \Tr \sigma_2 &\mathbb I_2 ) \\
+ \frac{\mathcal D_1}{2P} \Tr \sigma_1 \mathbb I_2 - \frac{\mathcal D_2}{2P} \Tr \sigma_2 \mathbb I_2 
+ \frac{\mathcal{D}_1 - \mathcal{D}_2}{2} &\mathbb I_2 = D(\delta u),
\end{align}
\end{subequations}
where we have introduced the following notation for $i=1,2$
\begin{align}
\mathcal{T}_{w,i} &= c_w \rho_w \lvert U_w - u_i \rvert \big[ (U_w - u_i) \cos \theta + (U_w - u_i)^\perp \sin \theta \big], \\
\mathcal{D}_i &= \lvert D(u_i) \rvert.
\end{align} 
Subsequently, by proceeding entirely analogously to section \ref{sec:wel-posedness-intermediate} (the only difference being that we use Lemma \ref{strainratelemma} in the case $\epsilon = 0$ rather than $\epsilon > 0$), we can obtain the analogous version of estimate \eqref{est:diff-total}, which is given by 
\begin{equation}
\begin{gathered}
\frac{1}{2} \frac{\rm d}{\dt} \bigg[ \lVert \delta u \rVert_{L^2}^2 + \lVert \nabla \delta u \rVert_{L^2}^2 + \mathcal{E}^{-1} \lVert \delta \sigma \rVert_{L^2}^2 + \mathcal{E}^{-1} (1 + \alpha^2) \lVert \nabla \delta \sigma \rVert_{L^2}^2 + \mathcal{E}^{-1} \alpha^2 \lVert \Delta \delta \sigma \rVert_{L^2}^2 \bigg] \\
\lesssim (1 + \Vert \sigma_1 \Vert_{H^2} + \Vert \sigma_2 \Vert_{H^2} + \Vert u_1 \Vert_{H^2} + \Vert u_2 \Vert_{H^2}) (\Vert \delta u\Vert_{H^1}^2 + \Vert \delta \sigma\Vert_ {H^2}^2 ).
\end{gathered} \end{equation}
From this estimate we can deduce the stability and uniqueness for the Voigt-EVP model. This finishes the proof of Theorem \ref{regularisedevplimit}.

\begin{remark}
We remark that the proof of the convergence of the solutions in the viscosity cutoff limit $\epsilon \rightarrow 0$, is in some sense related to the approach in \cite{berselli}, in which the local well-posedness of a model for non-Newtonian incompressible flows with power-law stress tensor was established uniformly in the cutoff parameter of the viscosity. 
\end{remark}

\section{Discussion and conclusion} \label{conclusion}
In this paper we commenced the rigorous analytical study of one of the fundamental models of sea-ice dynamics, namely the elastic-viscous-plastic (EVP) sea-ice model. We introduced a regularised version of the EVP model, by means of the Voigt regularisation, which is inspired by the original elastic regularisation of the Hibler model in \cite{hunke} (as both regularisations formally do not impact the asymptotic dynamics). We have shown this Voigt-EVP model to be globally well-posed. One interesting question, which we leave to future work, is whether the Voigt-EVP model also has a suitable notion of weak solution. 

Due to the (formal) linear ill-posedness of the EVP model as discussed in Section \ref{illposednesssection}, it seems necessary to introduce a regularised version of the EVP model of the type \eqref{limitEVP1}-\eqref{limitEVP3}, in order to obtain a well-posed sea-ice model. Finally, in order to conclude the proof of the global well-posedness we took the limit of the viscosity cutoff parameter $\epsilon$ going to zero (while keeping the Voigt regularisation), and obtain that the Voigt-EVP model is globally well-posed.

Moreover, we would like to mention that over time several related models to the EVP model have been developed in the sea-ice literature, see for example \cite{hunkelinearization,bouillon2009}. The approach developed here will also have potential applications in the rigorous study of several of these related sea-ice models.

We will make several remarks on straightforward extensions of our results, which have been omitted for reasons of brevity but can be treated by using the approach developed in this paper.
\paragraph{About the advection term}
If one includes an advection term in the momentum equation of the intermediate system \eqref{simplifiedmodel1}-\eqref{simplifiedmodel3}, by applying the same methods as in this paper, one can prove the local well-posedness of the intermediate system by using similar estimates as in this paper. Note however that the existence time will most likely depend on $\epsilon$ (unlike for Corollary \ref{cor:intermediatesys}). This therefore prohibits taking the limit $\epsilon \rightarrow 0$, as we have done in the proof of Theorem \ref{regularisedevplimit}.
\paragraph{Using the original strain rate}
We remark that the replacement of $\overline{\mathcal{D}}$ by $\lvert D(u) \rvert$ in the Voigt-EVP model is done merely for the sake of convenience. In fact, by using the approach from this paper one can prove the global well-posedness of the model \eqref{simplifiedmodel1}-\eqref{simplifiedmodel3} in which $\mathcal{D}_\epsilon$ has been replaced by
\begin{equation*}
\sqrt{\overline{\mathcal{D}}^2 + \epsilon^2} = \sqrt{\frac{2}{\overline e^2} \vert D(u) - \frac{1}{2} [\Tr D(u)] \mathbb I_2 \vert^2 + \vert \Tr D(u) \vert^2 + \epsilon^2}.
\end{equation*}
This can be done by using similar estimates as in Section \ref{sec:intermediatesys}, but by using the following inequalities
\begin{equation} \label{originalstrainrateestimate}
\left\lVert \sqrt{\overline{\mathcal{D}}^2 + \epsilon^2} \right\rVert_{L^2} \lesssim \lVert \nabla u \rVert_{L^2} + \epsilon, \quad \left\lVert \nabla  \left( \sqrt{\overline{\mathcal{D}}^2 + \epsilon^2} \right) \right\rVert_{L^2} \lesssim \lVert \Delta u \rVert_{L^2}.
\end{equation}
Subsequently, one can proceed as in the proof of Theorem \ref{regularisedevplimit} and send $\epsilon \rightarrow 0$ in order to obtain a solution of the Voigt-EVP model with the original strain rate.
\paragraph{Different regularisations of the strain rate}
We also note that, in addition to \eqref{viscosityregularisation}, different regularisations of the strain rate have been used in the sea-ice literature. For example, one can set maximal values for the shear and bulk viscosities, as is done in \cite{hibler}. This means that, instead of $\mathcal{D}_\epsilon$, one uses
\begin{equation} \label{hiblerregularisation}
\max \{ \overline{\mathcal{D}}, \epsilon \} = \frac{\overline{\mathcal{D}} + \epsilon}{2} + \frac{1}{2} \lvert \overline{\mathcal{D}} - \epsilon \rvert.
\end{equation}
One can choose to regularise this cutoff by the following function for some $\gamma \in (0, \epsilon)$
\begin{equation}
\overline{\mathcal{D}}_{\epsilon,\gamma} \coloneqq \frac{\overline{\mathcal{D}} + \epsilon + \gamma}{2} + \frac{1}{2} \sqrt{\lvert \overline{\mathcal{D}} - \epsilon \rvert^2 + \gamma^2}.
\end{equation}
Then one finds 
\begin{equation*}
\left\lVert \overline{\mathcal{D}}_{\epsilon,\gamma} \right\rVert_{L^2} \lesssim \lVert \nabla u \rVert_{L^2} + \gamma + \epsilon, \quad \left\lVert \nabla \overline{\mathcal{D}}_{\epsilon,\gamma} \right\rVert_{L^2} \lesssim \lVert \Delta u \rVert_{L^2}.
\end{equation*}
It is straightforward to repeat the proof of Theorem \ref{regularisedevplimit} to obtain a global strong solution of the Voigt-EVP model with the viscosity regularisation from equation \eqref{hiblerregularisation} by taking the limit $\gamma \rightarrow 0$. Moreover, one can also choose to regularise the viscosities by means of the hyperbolic tangent, as was introduced in \cite{lemieux2009}, see also \cite{chatta}, but again one can treat this regularisation by similar methods as in this paper. Generally speaking, one can treat any $C^2$ approximation of $\overline{\mathcal{D}}$ by relying on the approach of this paper.
\section*{Acknowledgements}
The authors would like to thank Elizabeth Hunke for stimulating discussions at Texas A\&M University, and for her guidance through the various models used by practitioners. The authors would also like to thank the anonymous referees for their many comments and suggestions, which have improved the paper. D.W.B. acknowledges support from the Cambridge Trust and the Cantab Capital Institute for Mathematics of Information. D.W.B. and E.S.T. have benefitted from the inspiring environment of the CRC 1114 ``Scaling Cascades in Complex Systems'', Project Number 235221301, Project C09, funded by the Deutsche Forschungsgemeinschaft (DFG). M.T. also acknowledges the funding by the DFG within the CRC 1114 ``Scaling Cascades in Complex Systems'', Project Number 235221301, Project B09. Moreover, this work was also supported in part by the DFG Research Unit FOR 5528 on Geophysical Flows. D.W.B. and M.T. would like to acknowledge the kind hospitality of the Department of Mathematics, Texas A\&M University, and M.T. also acknowledges the generous hospitality of the Department of Applied Mathematics and Theoretical Physics, University of Cambridge, where part of this work was completed.

\footnotesize
\bibliographystyle{acm}
\bibliography{main}

\end{document}